\documentclass[a4paper, 11pt]{article}
\usepackage{amsmath}
\usepackage{amssymb}
\usepackage{amsthm}
\usepackage{color}
\usepackage[width=460pt, height=650pt]{geometry}

\usepackage{soul}

\DeclareMathOperator{\N}{\mathbb{N}}
\DeclareMathOperator{\R}{\mathbb{R}}

\DeclareMathOperator{\cK}{\mathfrak{K}}

\DeclareMathOperator{\iii}{\mathtt{i}}
\DeclareMathOperator{\jjj}{\mathtt{j}}
\DeclareMathOperator{\kkk}{\mathtt{k}}
\DeclareMathOperator{\hhh}{\mathtt{h}}

\DeclareMathOperator{\bA}{\mathbf{A}}
\DeclareMathOperator{\bt}{\mathbf{t}}

\DeclareMathOperator{\eps}{\varepsilon}

\DeclareMathOperator{\GL}{GL}
\DeclareMathOperator{\SL}{SL}

\renewcommand{\phi}{\varphi}

\numberwithin{equation}{section}

\theoremstyle{plain}
\newtheorem{theorem}{Theorem}[section]
\newtheorem{condition}[theorem]{Condition}
\newtheorem{corollary}[theorem]{Corollary}
\newtheorem{proposition}[theorem]{Proposition}
\newtheorem{lemma}[theorem]{Lemma}

\newtheorem*{claim}{Claim}

\theoremstyle{remark}
\newtheorem{remark}[theorem]{Remark}

\theoremstyle{definition}

\title{Dynamically defined subsets of generic self-affine sets}
\author{Bal\'azs B\'ar\'any$^{1,}$\thanks{BB acknowledges support from grants OTKA K123782 and OTKA~FK134251.} \and Sascha
  Troscheit$^{2,}$\thanks{ST
was funded by the Austrian Research Fund (FWF) Grant M-2813.\newline \indent Both authors acknowledge support from Aktion \"Osterreich-Ungarn
103öu6.}}

\begin{document}
\maketitle
\begin{center}
  \vspace{-1em}
$^1)$
Budapest University of Technology and Economics, Department of Stochastics, MTA-BME Stochastics Research Group, 1521~Budapest, P.O. Box 91, Hungary.\vskip.5em
$^2)$ Faculty of Mathematics, University of Vienna, Oskar Morgenstern Platz 1, 1090 Wien, Austria.
\end{center}
\vspace{1em}

\begin{abstract}
  In dynamical systems, shrinking target sets and pointwise recurrent sets are two important classes of dynamically
  defined subsets. 
  In this article we introduce a mild condition on the linear parts of the affine mappings that
  allow us to bound the Hausdorff dimension of cylindrical shrinking target and recurrence sets.
  For generic self-affine sets in the sense of Falconer, that is by randomising the translation
  part of the affine maps, we prove that these bounds are sharp.
  These mild assumptions mean that our results significantly extend and complement the existing
  literature for recurrence on self-affine sets.
\end{abstract}

\section{Introduction}
The shrinking target problem in dynamical systems investigates the ``size'' of the set of 
points that recur to a collection of (shrinking) targets infinitely many times. Letting $(X,T,\mu)$ be
a dynamical system with invariant measure $\mu$ and a collection of (measurable) subsets
$(B_k)_{k\in\N}$, $B_k\subseteq X$ one investigates 
\[
  R((B_k)_k) = \{ x\in X : T^k(x)\in B_k \text{ for infinitely many } k\in\N\}.
\]
Similarly, given a function $\psi:\N\to \R^+$, the pointwise recurrent set is defined as
\[
  S(\psi) = \{ x\in X : T^k(x) \in B(x,\psi(k)) \text{ for infinitely many } k\in\N\}.
\]
Often these sets are dense in the original space $X$, as well as $G_\delta$, and so dimension
theory is used to classify the sizes of such sets. The Hausdorff dimension is the most appropriate
choice here, as dense $G_\delta$ sets have full dimension for, e.g. the packing-, Minkowski-, and Assouad-type
dimensions. 

The shrinking target problem was first investigated by Hill and Velani 
for Julia sets who analysed their Hausdorff dimension \cite{Hill1} and found a zero-one law for its
Hausdorff measure \cite{Hill2}.
The shrinking target problem has intricate links to number theory when using naturally arising sets in
Diophantine approximation as the shrinking targets. This has received a lot of attention over recent
years, see for instance \cite{Allen19, Barany17, Koivusalo18, Persson17, Reeve11} for shrinking
target sets and \cite{Baker20, Bugeaud03, Fan13, Fan06,
Kim16, Li14, Liao13, Liao13a} for related research. 

The literature of recurrence sets so far has focussed mostly on zero-one laws for conformal and one dimensional dynamics, such as $\beta$-transformations, see Tan and Wang \cite{Tan11}, and  Zheng and Wu \cite{Zheng20}.
For self-similar and self-conformal dynamics these questions were explored by Seuret and
Wang \cite{Seuret15}, who also gave a pressure formula for the Hausdorff dimension, as well as Baker and Farmer \cite{Baker21} who stated a zero-one law dependent on a convergence condition of the size of the neighbourhoods. Finally, and most recently, Kirsebom, Kude, and Persson \cite{Kirsebom21} studied linear maps on the $d$-dimensional torus.

The above works mostly concern dynamical systems in $\R^1$ or conformal dynamics and transitioning
to higher dimensional non-conformal dynamics
presents severe challenges.
To circumvent the extreme challenges that affinities pose, a common strategy is to ``randomise'' the
affine maps by considering typical translation parameter. 
This approach was first considered by Falconer in his seminal article \cite{Falconer88}, whose
conditions were significantly relaxed by Solomyak \cite{Solomyak98} and generalised by Jordan, Pollicott and Simon \cite{JPS}.

This typicality with respect to the translation parameter allows one to say more about the
regularity of the attractors and is a commonly employed strategy, see for example \cite{JPS}.
Using such randomisation, Koivusalo and Ram\'irez
\cite{Koivusalo18} gave an expression for the Hausdorff dimension of a self-affine
shrinking target problem. They show that for a fixed symbolic target with exponentially shrinking diameter
and well-behaved affine maps, the Hausdorff dimension is typically given by the zero of an appropriate
pressure function. Strong assumptions are made on the affine system, as well as the fixed target
and in this article we significantly improve upon their results.

We will show that for a large family of self-affine systems and dynamical targets with non-fixed
centres the Hausdorff dimension is given by the intersection of two pressures: one being the
standard self-affine pressure function, the other being an inverse lower pressure related to the target.
Crucially, we do not expect the target to be fixed and the inverse pressure to exist.

Our condition also allows us to investigate the dimensions of sets with a pointwise recurrence, a quantitative version of recurrence for self-affine dynamics. As far as we are aware, this is the first time this was attempted for non-conformal dynamics in higher dimensions.

\section{Results}

\subsection{Self-affine sets and symbolic space}
Let $\bA = \{A_1, A_2, \cdots, A_N\}$ be a collection of non-singular $d \times d$ contracting matrices.
Let $\bt = \{t_1, t_2, \cdots, t_N\}$ be a collection of $N$ vectors in $\R^d$.

Let $\{1,\cdots,N\}$ be a finite alphabet and write $\Sigma_n, \Sigma_*,
\Sigma$ for the union of words of length $n$, the union of all finite length words, and all
infinite words, respectively. For words $\iii\in\Sigma_n$ and $\jjj\in\Sigma$ we write $\iii = i_1 i_2 \cdots i_n$ and
$\jjj = j_1 j_2 \cdots$ to denote the individual letters of $\iii$ and $\jjj$. For a word
$\iii\in\Sigma_*$, let $|\iii|$ denote the length of $\iii$. For any two words $\iii,\jjj\in\Sigma$,
let us denote the common prefix by $\iii\wedge\jjj$, that is, $\iii\wedge\jjj:=i_1\cdots
i_{|\iii\wedge\jjj|}$ and $|\iii\wedge\jjj|:=\min\{k\geq1:i_k\neq j_k\}-1$. We adapt the notation that if $|\iii\wedge\jjj|=0$ then
$\iii\wedge\jjj:=\emptyset$. For two $\iii,\jjj\in\Sigma_*$, denote by $\iii\prec\jjj$ if $\jjj$ is a prefix of $\iii$, that is, $|\iii\wedge\jjj|=|\jjj|\leq|\iii|$. Let $\sigma:\Sigma\to\Sigma$ be the left-shift operator on $\Sigma$, i.e. $\sigma(\iii) =
\sigma (i_1 i_2 i_3 \cdots) = i_2 i_3 i_4 \cdots$.
Let $\Phi_{\bt}=\{f_i(x)=A_ix+t_i\}_{i=1}^N$ be an iterated function system formed by affine
maps on $\R^d$. For a finite word $\iii\in\Sigma_*$, let $A_{\iii}=A_{i_1}\cdots A_{i_n}$ and
$f_{\iii}=f_{i_1}\circ\cdots\circ f_{i_n}$. It is a classical result that 
there exists a unique non-empty compact set $\Lambda\subset\R^d$ such that
$$
\Lambda=\bigcup_{i=1}^Nf_i(\Lambda).
$$
To avoid singleton sets we assume that $N\geq 2$ throughout. Let us denote by $\pi=\pi_{\bt}$ the
natural projection from $\Sigma$ to the attractor of $\Phi_{\bt}$, that is,
$$
\pi_{\bt}(\iii)=\lim_{n\to\infty}f_{i_1}\circ\cdots\circ f_{i_n}(0)=\sum_{k=1}^\infty A_{i_1}\cdots A_{i_{k-1}}t_{i_k}.
$$
Clearly, $\pi_{\bt}(\iii)=f_{i_1}(\pi_{\bt}(\sigma\iii))$ and so 
\begin{equation*}
	\pi_{\bt}(\iii)-\pi_{\bt}(\jjj)=A_{\iii\wedge\jjj}\big(\pi_{\bt}(\sigma^{|\iii\wedge\jjj|}\iii)-\pi_{\bt}(\sigma^{|\iii\wedge\jjj|}\jjj)\big).
\end{equation*}

For a $d\times d$ matrix $A\in \GL_d(\R)$ let $\phi^t(A)$ be the usual singular value function defined by
\[
  \phi^t(A) = \begin{cases}
    \alpha_1(A)\alpha_2(A) \dots \alpha_{\lfloor t\rfloor}(A) \alpha_{\lfloor t
    \rfloor+1}(A)^{t- \lfloor t \rfloor}&\text{for }0\leq t <d,\\
    \left(\alpha_1 (A) \dots \alpha_{d-1}(A) \alpha_d(A)\right)^{t/d}& \text{for } t \geq d.
\end{cases}
\]
For any ball $B$, clearly, $A(B)$ is an ellipsoid and, as it was shown in \cite[Proof of Proposition~5.1]{Falconer88}, it can be covered by at most $(4|B|)^{d}\frac{\alpha_1(A)\cdots\alpha_{\lfloor s\rfloor}(A)}{\alpha_{\lceil s\rceil}(A)^{\lfloor s\rfloor}}$-many cubes with side length $\alpha_{\lceil s\rceil}(A)$.

The pressure  of the self-affine system is defined as
\[
P(t) = \lim_{n\to\infty} \frac{1}{n} \log\sum_{\iii\in\Sigma_n}
\phi^t(A_{\iii}),
\] 
where we note that this limit exists because of the subadditivity of $\phi^t(A)$. Further,
the pressure is continuous in $t$, strictly decreasing, and satisfies $P(0)=\log N$ and $P(t)\to-\infty$ as
$t\to\infty$. 

Throughout the paper we will use the following extra condition:
\begin{condition}\label{thm:injection}
	Assume that $\bA$ is such that for every $s>0$ there exists $C>0$ and $K\in\N$ such that for every
	$\iii,\jjj\in\Sigma_*$ there exists $\kkk\in\Sigma_K$ with
	\[
	\phi^s(A_{\iii\kkk\jjj}) \geq
	C\phi^s(A_{\iii})\phi^s(A_{\jjj}).
	\]
\end{condition}

Similar conditions has been introduced earlier by Feng \cite{Feng09} and K\"aenm\"aki and Morris
\cite{Kaenmaki18}. Feng \cite[Proposition~2.8]{Feng09} showed that under a mild irreducibility
condition there exists $C>0$ and $K>0$ such that for every $\iii,\jjj\in\Sigma_*$ there exists
$\kkk$ with $|\kkk|\leq K$ such that $\|A_{\iii\kkk\jjj}\|\geq C\|A_{\iii}\|A_{\jjj}\|$. Later, this
inequality was generalised by K\"aenm\"aki and Morris \cite[Lemma~3.5]{Kaenmaki18} for the singular
value function under more restrictive but natural irreducibility conditions. Unfortunately, the
uncertainty of the length of the "buffer" word $\kkk$ in the previous conditions does not allow us
to study shrinking target and recurrence sets effectively. We will show in Section~\ref{sec:irred} and Section
\ref{sec:condition} that under some irreducibility and proximality assumptions,
Condition~\ref{thm:injection} holds.

\subsection{Shrinking targets}

Let $(\lambda_k)_{k\in\N} \in(\Sigma_*)^{\N}$ be a sequence of target cylinders.
We are interested in the shrinking target set 
\[
  R_{\bt}( (\lambda_k)_{k\in\N}) = \pi_{\bt} \left\{ \iii\in\Sigma : \sigma^k \iii\in [\lambda_k] \text{ for
  infinitely many }k\in\N\right\}.
\]

For our sequence of target cylinders, we define the following inverse lower pressure:
\begin{equation}\label{eq:inverselower}
  \alpha(t) = \liminf_{k\to\infty}\frac{-1}{k}\log \phi^t(A_{\lambda_k})
\end{equation}
Let 
\begin{equation}\label{eq:defs0}
	s_0:=\inf\{t>0:P(t)\leq\alpha(t)\}.
\end{equation}
If $\liminf_{n\to\infty}\frac{|\lambda_k|}{k}<\infty$ then there exists a unique solution $s_0$ to
the equation $P(s_0)=\alpha(s_0) \geq 0$, see Lemma \ref{thm:uniqueintersection}. 
Otherwise $s_0=0$. We prove that this value gives the Hausdorff dimension of the shrinking target
set under some assumptions on the matrices $\bA$.

\begin{theorem}
  \label{thm:mainTheorem}
  Let $\bA$ be a collection of $d\times d$ matrices. Suppose that $\bA$ satisfies Condition~\ref{thm:injection} and $\|A\|<1/2$ for all $A\in\bA$. Then
  $$
  \dim_HR_{\bt}( (\lambda_k)_k)=\min\{d,s_0\}\text{  for Lebesgue-almost
  	every $\bt$. }
  $$
  Moreover, $\mathcal{L}_d(R_{\bt}( (\lambda_k)_k))>0$ for Lebesgue-almost every $\bt$ if $s_0>d$.
\end{theorem}

\begin{remark}
  The upper dimension bounds do not just hold for almost every translation $\bt$, but hold for all
  translations.
\end{remark}

Similar result has been obtained by Koivusalo and Ram\'irez \cite{Koivusalo18} for shrinking targets
on self-affine sets. Firstly, they assume that there exists a constant $C>0$ such that for every
$\iii,\jjj\in\Sigma_*$, $\varphi^s(A_{\iii}A_{\jjj})\geq C\varphi^s(A_{\iii})\varphi^s(A_{\jjj})$,
secondly, they assume that $\alpha(t)$ is taken as a limit. The first condition holds only for a
restrictive family of matrices, see Remark~\ref{rem:compare}. By using a more detailed analysis on the pressure function, we were able to relax the condition on the limit as well.

\subsection{Recurrence sets}

Now, we turn our attention to the recurrence sets. Let $\psi\colon\N\mapsto\N$, and let $\beta=\liminf_{n\to\infty}\frac{\psi(n)}{n}$. Consider the set
$$
S_{\bt}(\psi):=\pi_{\bt}\left\{\iii\in\Sigma:\sigma^k\iii\in[\iii|_{\psi(k)}]\text{ for infinitely many $k\in\N$}\right\}.
$$
Let us define the square-pressure function
$$
P_2(t)=\lim_{n\to\infty}\frac{-1}{n}\log\sum_{\iii\in\Sigma_n}\left(\varphi^t(A_{\iii})\right)^2.
$$
Note that the limit exists again because of the subadditivity of $\phi^t(A)$. Further, the pressure is continuous in $t$, strictly increasing, and satisfies $P_2(0)=-\log N$ and $P_2(t)\to\infty$ as
$t\to\infty$.

\begin{theorem}
	\label{thm:mainTheorem2}
	Let $\bA$ be a collection of $d\times d$ matrices. Suppose that $\bA$ satisfies
	Condition~\ref{thm:injection} and $\|A\|<1/2$ for all $A\in\bA$. Let $\psi:\N\to\N$ with
	$\beta:=\liminf\psi(n)/n<1$ then
	$$
	\dim_HS_{\bt}(\psi)=\min\{d,r_0\}\text{  for Lebesgue-almost
		every $\bt$, }
	$$
	where $r_0$ is the unique solution of the equation
	\begin{equation}\label{eq:rootsquare}
	(1-\beta)P(r_0)=\beta P_2(r_0).
	\end{equation}
	Moreover, $\mathcal{L}_d(S_{\bt}(\psi)>0$ for Lebesgue-almost every $\bt$ if $r_0>d$.
\end{theorem}

The equation \eqref{eq:rootsquare} applies specifically only to the case when $\beta\leq1$, for other values of $\beta$ it needs to be modified accordingly. The condition $\beta<1$ is purely technical and relies on the fact that the buffer word in Condition~\ref{thm:injection} depends on both of the words before and after it. Hence, for recurrence rates greater than $1$ it might cause ``self-dependence'' in the buffer word, which then may not exist. We note that under the stronger assumption on the matrices by Koivusalo and Ram\'irez \cite{Koivusalo18}, Theorem~\ref{thm:mainTheorem2} can be generalized for any value $\beta\in[0,\infty]$ with a straightforward modification of \eqref{eq:rootsquare} and the proof of Theorem~\ref{thm:mainTheorem2}.

\subsection{Irreducibility of matrices}\label{sec:irred}
Let us denote by $\wedge^k\R^d$ the \textit{$k$-th exterior product of $\R^d$}. For $A\in
\GL_d(\R)$, we can define an invertible linear map $A^{\wedge k}\colon\wedge^k\R^d\mapsto\wedge^k\R^d$ by setting
$$
A^{\wedge k}(u_1\wedge\cdots\wedge u_k)=(A u_1)\wedge\cdots\wedge(Au_k).
$$
Let us consider the following tensor product of the exterior algebras
$$
\widehat{W}=\wedge^1\R^d\otimes\cdots\otimes\wedge^{d-1}\R^d.
$$
Again, for $A\in \GL_d(\R)$, we can define an invertible linear map
$\widehat{A}\colon\widehat{W}\mapsto\widehat{W}$ by setting for $u=u_1\otimes\cdots\otimes u_{d-1}$,
$$
\widehat{A}(u_1\otimes\cdots\otimes u_{d-1})=(A^{\wedge1}u_1)\otimes\cdots\otimes(A^{\wedge(d-1)}u_{d-1}).
$$
We define a linear subspace $W$ of $\widehat{W}$, which is generated by the flags of $\R^d$ as follows:
$$
W=\mathrm{span}\{u_1\otimes(u_1\wedge u_2)\otimes\cdots\otimes(u_1\wedge\cdots\wedge u_{d-1}):\{u_1,\ldots,u_{d-1}\}\text{ linearly independent in }\R^d\}.
$$
We call $W$ the \textit{flag vector space}. Note that the flag space $W$ is
invariant with respect to the linear map $\widehat{A}$ for $A\in \GL_d(\R)$.

We say that $A\in \GL_d(\R)$ is \textit{fully proximal} if it has $d$ distinct eigenvalues in
absolute value. Note that $A$ is fully proximal if and only if $A^{\wedge k}$ is $1$-proximal for
every $k$ if and only if $\widehat{A}$ is $1$-proximal on $W$. We say that the tuple $\bA$ is fully
proximal if there exists a finite product $A_{i_1}\cdots A_{i_k}$ formed by the elements in $\bA$,
which is fully proximal.

We say that the tuple $\bA$ is \textit{fully strongly irreducible} or \textit{strongly irreducible
over $W$} if there are no finite collections $V_1,\ldots,V_n$ of proper subspaces of $W$ such that
$$
\bigcup_{A\in\bA}\bigcup_{k=1}^n\widehat{A}V_k=\bigcup_{k=1}^nV_k.
$$

\begin{proposition}\label{prop:multicond}
	Let $\bA$ be a tuple of matrices in $\GL_d(\R)$ such that $\bA$ is fully proximal and fully
	strongly irreducible. Then for every $0<s<d$ there exists $C>0$ and $K\in\N$ such that for every
	$\iii,\jjj\in\Sigma_*$ there exists $\kkk\in\Sigma_K$ with
	\begin{equation*}
		\phi^s(A_{\iii\kkk\jjj}) \geq	C\phi^s(A_{\iii})\phi^s(A_{\jjj}).
	\end{equation*}
\end{proposition}

\begin{remark}\label{rem:compare}
	Koivusalo and Ram\'irez \cite{Koivusalo18} assumed that there exists a constant
	$D>0$ such that for every $\iii,\jjj\in\Sigma_*$
	$$
	\varphi^s(A_{\iii} A_{\jjj})\geq D\varphi^s(A_{\iii})\varphi^s(A_{\jjj}).
	$$
	B\'ar\'any, K\"aenm\"aki and Morris \cite[Corollary~2.5]{BKM} showed that this condition for planar
	matrix tuples $\bA$ is equivalent with the following: $\bA$ can be decomposed into two sets $\bA_e$ and
	$\bA_h$ such that $\bA_e$ is strongly conformal (i.e. can be transformed into orthonormal matrices
	with a common base transformation) and if $\bA_h\neq\emptyset$, then $\bA_h$ has a strongly
	invariant multicone $\mathcal{C}$ (i.e.
	$\bigcup_{A\in\bA_h}A\overline{\mathcal{C}}\subset\mathcal{C}^o$) such that $A\mathcal{C} =
	\mathcal{C}$ for all $A\in\bA_e$. 
	
	Assuming fully strong irreducibilty and fully proximality is clearly a less restrictive requirement. For instance, in case of planar matrices fully strong irreducibility and fully
	proximality is equivalent with strong irreducibility and proximality.
\end{remark}

Using Proposition \ref{prop:multicond} we obtain the following immediate corollaries.
\begin{corollary}
	\label{cor:mainTheorem1b}
	Let $\bA$ be a collection of $d\times d$ matrices. Suppose that $\bA$ is fully strongly
	irreducible and fully proximal and $\|A\|<1/2$ for all $A\in\bA$. Then
	$$
	\dim_HR_{\bt}( (\lambda_k)_k)=\min\{d,s_0\}\text{  for Lebesgue-almost
		every $\bt$, }
	$$
	where $s_0$ is defined in \eqref{eq:defs0}.	Moreover, $\mathcal{L}_d(R_{\bt}( (\lambda_k)_k))>0$ for Lebesgue-almost every $\bt$ if $s_0>d$.
\end{corollary}

\begin{corollary}
	\label{cor:mainTheorem2b}
	Let $\bA$ be a collection of $d\times d$ matrices. Suppose that $\bA$ is fully strongly
	irreducible and fully proximal and $\|A\|<1/2$ for all $A\in\bA$. If $\beta<1$ then
	$$
	\dim_HS_{\bt}(\psi)=\min\{d,r_0\}\text{  for Lebesgue-almost
		every $\bt$, }
	$$
	where $r_0$ is the unique solution of the equation \eqref{eq:rootsquare}. Moreover, $\mathcal{L}_d(S_{\bt}(\psi))>0$ for Lebesgue-almost every $\bt$ if $r_0>d$.
\end{corollary}

\paragraph{Structure.}
We prove Theorem \ref{thm:mainTheorem} in Section~\ref{sec:proofst} and Theorem~\ref{thm:mainTheorem2} in Section~\ref{sec:proofr} using Condition \ref{thm:injection}.
First, we derive elementary results on the inverse lower pressure $\alpha$ defined in \eqref{eq:inverselower} in
Section~\ref{sec:basics}. We will also recall results about the pressure $P$ and prove the uniqueness of the solution of $P(s_0) =
\alpha(s_0)$. We proceed in Section~\ref{sec:upperBound} by proving the upper bound to
Theorem~\ref{thm:mainTheorem} and finish the lower bound proof in Section~\ref{sec:lowerBound} with
an energy estimate. Similarly, Section~\ref{sec:ubrec} is devoted to show the upper bound and Section~\ref{sec:lbrec} is to show the lower bound of Theorem~\ref{thm:mainTheorem2}. Section~\ref{sec:condition} contains the proof of Proposition~\ref{prop:multicond}, which shows that the assumptions in Corollary~\ref{cor:mainTheorem1b} and \ref{cor:mainTheorem2b} are sufficient.

\section{Dimension of shrinking targets}\label{sec:proofst}

\subsection{Basic properties and the inverse lower pressure function}\label{sec:basics}

Let $(\lambda_k)_{k\in\N} \in(\Sigma_*)^{\N}$ be a sequence and let $\alpha$ be the corresponding inverse lower pressure defined in \eqref{eq:inverselower}.

\begin{lemma}
  \label{thm:zeroEverywhere}
  If $\alpha(s)=0$ for some $s>0$, then $\liminf_{n\to\infty} |\lambda_n|/n =0$.
  Conversely, if\linebreak $\liminf_{n\to\infty}|\lambda_n|/ n =0$, then $\alpha(t)=0$ for all $t>0$.

  In particular, if there exists $s>0$ such that $\alpha(s)=0$, then $\alpha(t)=0$ for all $t\geq
  0$.
\end{lemma}
\begin{proof}
  Let $\gamma = \max_{i\in\Sigma_1}\{\alpha_1(A_i)\}$ and $\underline\gamma =
  \min_{i\in\Sigma_1}\{\alpha_d(A_i)\}$.
  Observe that by definition 
  \begin{equation}\label{eq:bounds}
    \underline\gamma^{s|\iii|}\leq\phi^s(A_{\iii}) \leq \gamma^{s|\iii|}.
  \end{equation}
  Assume that $\alpha(s)=0$. Since $-1/n \log \phi^s(A_{\lambda_n})\geq 0$, this implies that there is a
  subsequence $n_k$ such that $1/n_k \log \phi^s(A_{\lambda_{n_k}})\nearrow 0$. But then $1/n_k \log
  \gamma^{s|\lambda_{n_k}|} = s|\lambda_{n_k}|/n_k \log \gamma \nearrow 0$ and so
  $|\lambda_{n_k}|/n_k\searrow 0$, as required.

  For the other direction assume $|\lambda_{n_k}|/n_k \to 0$ for some subsequence $n_k$. Then, for any
  $t \geq 0$,
  \[
    \alpha(t) = \liminf_{n\to\infty}-\frac{1}{n}\log\phi^t(A_{\lambda_n})
    \leq \liminf_{n\to\infty}-\frac{1}{n}\log \underline\gamma^{t|\lambda_n|}
    \leq \liminf_{k\to\infty}\frac{|\lambda_{n_k}|}{n_k}t(-\log\underline\gamma)
    \leq 0.
  \]
  Combining this with the trivial inequality $\alpha(t)\geq 0$ we get the desired conclusion that
  $\alpha(t)=0$ for all $t\geq 0$.
\end{proof}
Similarly, if the modified pressure function is extremal in the other direction it must be extremal
everywhere. 
\begin{lemma}
  \label{thm:upperLimit}
 If $\alpha(s)=\infty$ for some $s>0$, then $\lim_{n\to\infty} |\lambda_n|/n =\infty$.
  Conversely, if\linebreak $\lim_{n\to\infty}|\lambda_n|/ n =\infty$, then $\alpha(t)=\infty$ for every $t>0$.

  In particular, if there exists $s>0$ such that $\alpha(s)=\infty$, then $\alpha(t)=\infty$ for all $t>0$.
\end{lemma}
The proof is analogous to that of Lemma \ref{thm:zeroEverywhere} and is left to the reader.

\begin{lemma}\label{thm:alfaprop}
  Assume that $0<\liminf_{k\to\infty}|\lambda_k|/k<\infty$.
  Then the function $\alpha(t)$ is strictly monotone increasing, and continuous in $t$. Moreover, $\alpha(t)\to\infty$ as $t\to \infty$ and $\alpha(0) = 0$.
\end{lemma}

\begin{proof}
Note that by Lemma~\ref{thm:zeroEverywhere} and Lemma~\ref{thm:upperLimit} the inverse lower pressure satisfies $0<\alpha(t)<\infty$
for all $t> 0$.  Letting $t=0$, we have
\[
\alpha(0) = \liminf_{n\to\infty}-\frac{1}{n}\log\phi^{0}(A_{\lambda_k})
=\log 1 =0.
\]

For every $k\in\N$,
$$
\frac{-1}{k}\log\varphi^s(A_{\lambda_k})\leq\frac{|\lambda_k|}{k}s(-\log\underline{\gamma}),
$$
and so
\begin{equation}\label{eq:liminflb}
s\liminf_{k\to\infty}\frac{|\lambda_k|}{k}\geq\frac{\alpha(s)}{-\log\underline{\gamma}}\text{ for every }s>0.
\end{equation}
This shows that $\alpha(t)$ is continuous at $t=0$.

For any $t>0$ and $\varepsilon>0$ sufficiently small we have
\begin{align*}
\alpha(t-\eps)     &\leq \liminf_{k\to\infty} - \frac{1}{k}\log\left(
\alpha_1(A_{\lambda_k})^{-\eps}\phi^{t}(A_{\lambda_k})\right)\\
&\leq \liminf_{k\to\infty} - \frac{1}{k} \log\left(
\gamma^{-\eps |\lambda_{k}|}\phi^t(A_{\lambda_k})\right)\\
&= \liminf_{k\to\infty}  \frac{1}{k} \left(
\eps|\lambda_k|\log\gamma-\log\phi^t(A_{\lambda_k})\right)\\
&\leq \limsup_{k\to\infty}  \frac{\eps|\lambda_k|\log\gamma}{k}
+\liminf_{k\to\infty}\left(-\frac{1}{k}\log\phi^t(A_{\lambda_k})\right)\\
&=
\alpha(t) - \eps\frac{\log\gamma}{(t-\varepsilon)\log\underline{\gamma}}\alpha(t-\varepsilon),
\end{align*}
where in the last inequality we applied \eqref{eq:liminflb} with $s=t-\varepsilon$. Hence,
\begin{equation}\label{eq:bound1}
\alpha(t-\varepsilon)\leq \alpha(t)\left(1+\frac{\varepsilon\log\gamma}{(t-\varepsilon)\log\underline{\gamma}}\right)^{-1},
\end{equation}
which shows that $\alpha(t)$ is strictly monotone increasing on $(0,\infty)$.

For an $s>0$, let $n_k(s)$ be a sequence for which the lower limit in $\alpha(s)$ is achieved. Then by \eqref{eq:bounds},
\[
\frac{-|\lambda_{n_k(s)}|s\log\gamma}{n_k(s)}\leq\frac{-1}{n_k(s)}\log\varphi^s(A_{\lambda_{n_k(s)}}).
\]
Hence for every $s>0$, 
\[
\limsup_{k\to\infty}\frac{|\lambda_{n_k(s)}|}{n_k(s)}\leq \frac{\alpha(s)}{-s\log\gamma}.
\]
This implies that
\begin{align*}
\alpha(t-\eps) &
=\lim_{k\to\infty}-\frac{1}{n_k(t-\eps)}\log\phi^{t-\eps}(A_{\lambda_{n_k(s)}})\\
& \geq \liminf_{k\to\infty} -\frac{1}{n_k(t-\eps)}\log \left(
\alpha_d(A_{\lambda_{n_k(t-\eps)}})^{-\eps}\phi^t(A_{\lambda_{n_k(t-\eps)}})\right)\\
& \geq
\liminf_{k\to\infty} - \frac{1}{n_k(t-\eps)}\log\left( \underline{\gamma}^{-\eps |\lambda_{n_k(t-\eps)}|} \varphi^t(A_{\lambda_{n_k(t-\eps)}})\right)\\
&=\liminf_{k\to\infty}\left(-\frac{1}{n_k(t-\eps)}\log\varphi^t(A_{\lambda_{n_k(t-\eps)}}) +
\eps\frac{|\lambda_{n_k(t-\eps)}|}{n_k(t-\eps)}\log\underline{\gamma}\right)\\
&\geq\liminf_{k\to\infty}\left(-\frac{1}{n_k(t-\eps)}\log\varphi^t(A_{\lambda_{n_k(t-\eps)}})\right) +
\eps\log\underline{\gamma}\limsup_{k\to\infty}\frac{|\lambda_{n_k(t-\eps)}|}{n_k(t-\eps)}\\
&\geq \alpha(t) - \eps\frac{\alpha(t-\varepsilon)\log\underline{\gamma}}{(t-\varepsilon)\log\gamma}.
\end{align*}
Thus,
$$
\alpha(t-\eps)\geq\alpha(t)\left(1+\frac{\eps\log\underline{\gamma}}{(t-\eps)\log\gamma}\right)^{-1},
$$
which together with \eqref{eq:bound1} implies continuity.

 To show the limit as $t\to\infty$, observe that 
\begin{equation*}
	\alpha(t) = \liminf_{k\to\infty}-\frac{1}{k}\log\phi^{t}(A_{\lambda_k})
	\geq \liminf_{k\to\infty} - \frac{1}{k}\log  \gamma^{t |\lambda_{k}|}
	\geq t(-\log\gamma)\liminf_{k\to\infty}\frac{|\lambda_k|}{k}.
	\qedhere
\end{equation*}
\end{proof}

\begin{lemma}
  \label{thm:uniqueintersection}
  Assume that $\liminf_{k\to\infty}|\lambda_k|/k<\infty$.
  Then there exists a unique $s_0>0$ such that $P(s_0)=\alpha(s_0)$. 
  Further, $P(s_0)\geq 0$.
\end{lemma}
\begin{proof}
  If $\liminf_{k\to\infty}|\lambda_k|/k>0$ then the first statement follows by Lemma~\ref{thm:alfaprop} since $P(0)-\alpha(0)=\log N$, and $P(t)-\alpha(t)\to-\infty$ as $t\to\infty$ and $P(t)-\alpha(t)$ is  strictly motonone decreasing. If $\liminf_{k\to\infty}|\lambda_k|/k=0$ then by Lemma~\ref{thm:zeroEverywhere} $\alpha(t)\equiv0$ for $t\geq0$ and then the uniqueness of the solution follows by the uniqueness of the root of $P$.

  The second conclusion follows from the observation that $\alpha(t)\geq 0$ for all $t\geq 0$.
\end{proof}

The following lemma is standard, but we include it for completeness.
\begin{lemma}
  \label{thm:singularratio}
  Let $\iii\in\Sigma_*$ be a finite word and let $0<t < s$. Then,
  \begin{equation}
  \label{eq:inequalityratio}
\frac{\phi^s(A_{\iii})}{\phi^t(A_{\iii})} \leq \gamma^{(s-t)|\iii|} 
\end{equation}
  for some uniform $0<\gamma<1$.
\end{lemma}
\begin{proof}
  For $0<t<s<d$,
  \begin{align*}
    \frac{\phi^s(A_{\iii})}{\phi^t(A_{\iii})} 
    &=\frac{\alpha_1(A_{\iii})\dots \alpha_{\lfloor s\rfloor}(A_{\iii})\cdot\alpha_{\lfloor s\rfloor+1}(A_{\iii})^{s-\lfloor
    s\rfloor}}
    {\alpha_1(A_{\iii})\dots \alpha_{\lfloor t\rfloor}(A_{\iii})\cdot\alpha_{\lfloor t\rfloor+1}(A_{\iii})^{t-\lfloor
    t\rfloor}}\\
    &\leq
    \frac{\alpha_1(A_{\iii})\dots \alpha_{\lfloor t\rfloor}(A_{\iii})\cdot\alpha_{\lfloor
      t\rfloor+1}(A_{\iii})^{\lfloor s\rfloor - \lfloor t\rfloor}\alpha_{\lfloor t\rfloor+1}(A_{\iii})^{s-\lfloor
    s\rfloor}}
    {\alpha_1(A_{\iii})\dots \alpha_{\lfloor t\rfloor}(A_{\iii})\cdot\alpha_{\lfloor t\rfloor+1}(A_{\iii})^{t-\lfloor
    t\rfloor}}\\
    &= \alpha_{\lfloor t\rfloor+1}(A_{\iii})^{s-t}.
  \end{align*}
  Similarly,
  for $0<t<d \leq s$,
  \begin{align*}
    \frac{\phi^s(A_{\iii})}{\phi^t(A_{\iii})} 
    &=\frac{(\alpha_1(A_{\iii})\dots \alpha_{d}(A_{\iii}))^{s/d}}
    {\alpha_1(A_{\iii})\dots \alpha_{\lfloor t\rfloor}(A_{\iii})\cdot\alpha_{\lfloor t\rfloor+1}(A_{\iii})^{t-\lfloor
    t\rfloor}}\\
    &=\left(\alpha_1(A_{\iii})\dots \alpha_{\lfloor t\rfloor}(A_{\iii})\right)^{s/d-1}\alpha_{\lfloor t\rfloor+1}(A_{\iii})^{s/d-t+\lfloor
    	t\rfloor}\left(\alpha_{\lfloor t\rfloor+2}(A_{\iii})\cdots\alpha_{d}(A_{\iii})\right)^{s/d}\\\
    &\leq\alpha_1(A_{\iii})^{(s/d-1)\lfloor t\rfloor+s/d-t+\lfloor t\rfloor +s(d-\lfloor t\rfloor-1)/d}=\alpha_1(A_{\iii})^{s-t}.
  \end{align*}
  Finally, for $d\leq t <s$,
  \begin{align*}
    \frac{\phi^s(A_{\iii})}{\phi^t(A_{\iii})} 
    &=\frac{(\alpha_1(A_{\iii})\dots \alpha_{d}(A_{\iii}))^{s/d}}
    {(\alpha_1(A_{\iii})\dots \alpha_{d}(A_{\iii}))^{t/d}}
    = (\det(A_{\iii}))^{(s-t)/d}.
  \end{align*}
  We conclude that \eqref{eq:inequalityratio} holds for $\gamma :=
  \max_{i\in\Sigma_1}\{\alpha_1(A_i)\}$ by submultiplicativity.
\end{proof}

\subsection{Upper bound to Theorem \ref{thm:mainTheorem}}
\label{sec:upperBound}
Note that $R_{\bt}( (\lambda_k)_k)$ is a $\limsup$ set that can be written as
\[
  R_{\bt}( (\lambda_k)_k) = \bigcap_{k_0=1}^\infty \bigcup_{k=k_0}^\infty
  \bigcup_{\iii\in\Sigma_k} \pi_{\bt}([\iii \lambda_k]).
\]

Temporarily fix $t\geq 0$. By definition, for every $\delta>0$ there exists $k_0$ large enough such that 
\[
  -\frac{1}{k}\log\phi^t(A_{\lambda_k}) \geq \alpha(t)-\delta
\]
for all $k\geq k_0$.
This can be rearranged to give
\[
   \phi^t(A_{\lambda_k})\leq e^{-k(\alpha(t)-\delta)}.
\]
Similarly, for every $\delta>0$, we obtain
\[
  \sum_{\iii\in\Sigma_k}\phi^t(A_{\iii}) \leq e^{k(P(t)+\delta)}
\]
for large enough $k$.
For the lower bounds, we note that for all $\delta>0$ there exists a subsequence $k_n$ such that
\begin{equation}
  \label{eq:targetEstimate}
  \phi^t(A_{\lambda_{k_n}}) \geq e^{-k_n (\alpha(t)+\delta)}
\end{equation}
and for large enough $k$,
\begin{equation}
  \label{eq:pressureEstimate}
  \sum_{\iii\in\Sigma_k}\phi^t(A_{\iii}) \geq e^{k(P(t)-\delta)}
\end{equation}
by submultiplicativity and existence of the limit.

Assume that $\liminf_k|\lambda_k|/k<\infty$.
Let $s>s_0$ and note that $P(s)-\alpha(s)<0$. We set $\delta>0$ small enough such that
$\eta:=P(s)-\alpha(s)+2\delta <0$. Let $B$ be a ball with sufficiently large radius such that $f_i(B)\subset B$ for all $i=1,\ldots,N$. Hence, by using the cover given in \cite[Proof of Proposition~5.1]{Falconer88} we obtain
\begin{align*}
  \mathcal{H}^s(R_{\bt}( (\lambda_k)_k)) 
  &\leq \liminf_{k_0\to\infty}
  \sum_{k=k_0}^{\infty}\sum_{\iii\in\Sigma_k}
  \phi^s(A_{\iii\lambda_k})(4|B|)^{d}\\
  &\leq
  \liminf_{k_0\to\infty}\sum_{k=k_0}^{\infty}\sum_{\iii\in\Sigma_k}\phi^s(A_{\iii})\phi^s(A_{\lambda_k})(4|B|)^{d}\\
  &\leq
  \liminf_{k_0\to\infty}\sum_{k=k_0}^{\infty}e^{-k(\alpha(s)-\delta)}\sum_{\iii\in\Sigma_k}\phi^s(A_{\iii})(4|B|)^{d}\\
  &\leq \liminf_{k_0\to\infty}\sum_{k=k_0}^\infty e^{k (P(s)+\delta-\alpha(s)+\delta)}(4|B|)^{d}\\
  &=\liminf_{k_0\to\infty}\sum_{k=k_0}^{\infty}e^{\eta k}(4|B|)^{d}
  \leq \liminf_{k_0\to\infty}\frac{e^{\eta k_0}}{1-e^{\eta}}(4|B|)^{d} = 0.
\end{align*}
Since $s>s_0$ was arbitrary, we conclude that $\dim_H R_{\bt}( (\lambda_k)_k) \leq s_0$ for all $\bt$.

Finally, consider the case when $\liminf_k |\lambda_k|/k = \infty$.
Let $s>0$ be arbitrary and again write $\gamma = \max_{i\in\Sigma_1}\{\alpha_1(A_i)\}$. Recall that
$\#\Sigma_1=N$ and observe that there exists $M$ such that
$|\lambda_k| \geq 2k\log N/(s \log \gamma^{-1})$ for $k\geq M$. Therefore $\gamma^{s|\lambda_k|}\leq
N^{-2k}$ for large enough $k$. The Hausdorff measure bound above becomes
\begin{align*}
  \mathcal{H}^s(R_{\bt}( (\lambda_k)_k)) &\leq \liminf_{k_0\to\infty}
  \sum_{k=k_0}^{\infty}\sum_{\iii\in\Sigma_k}
  \phi^s(A_{\iii\lambda_k})(4|B|)^{d}\\
  &\leq
  \liminf_{k_0\to\infty}\sum_{k=k_0}^{\infty}\sum_{\iii\in\Sigma_k}\phi^s(A_{\iii})\phi^s(A_{\lambda_k})(4|B|)^{d}\\
  &\leq
  \liminf_{k_0\to\infty}\sum_{k=k_0}^{\infty}\gamma^{|\lambda_k|s}\sum_{\iii\in\Sigma_k}\phi^s(A_{\iii})(4|B|)^{d}\\
  &\leq \liminf_{k_0\to\infty}\sum_{k=k_0}^\infty N^{-2k} N^k(4|B|)^{d}\\
  &=\liminf_{k_0\to\infty}\sum_{k=k_0}^{\infty}N^{-k}(4|B|)^{d} =0.
\end{align*}
As $s>0$ was arbitrary, this shows that $\dim_H R_{\bt}( (\lambda_k)_k) =0$ for all $\bt$.
\qed

\subsection{Lower bound to Theorem \ref{thm:mainTheorem}}
\label{sec:lowerBound}

To simplify the exposition we will abuse notation slightly and write $\phi^s(\iii)$ instead of
$\phi^s(A_{\iii})$ for $\iii\in\Sigma_*$.
  
For every sufficiently large $p\in\N$ and $s<\min\{s_0,d\}$, we construct a measure $\nu^s_p$ on the symbolic space $\Sigma$ and investigate its projection under the self-affine
iterated function system.
Let $m_k$ be a sequence on which the lower limit in $\alpha(s)$ is achieved and take a very sparse
subsequence such that 
\begin{equation}
  \label{eq:growthcondition}
  \sum_{k=1}^n m_k\leq (1+2^{-n})m_n
  \quad
  \text{and}
  \quad
  m_n \geq 2^{n}\sum_{i=1}^{n-1}(|\lambda_{m_i}|+K),
\end{equation}
where $K$ is the length of the buffer word defined in Condition~\ref{thm:injection}. We may further assume, without loss of generality, that $m_1 \gg
p$ and that $m_k \geq 2^k$. By the pigeonhole principle there exists $1\leq \widehat{p}_0 \leq p+K$ such that $m_k = \widehat{p}_0+(K+p)q$ for
infinitely many $q$. Again, by taking subsequences, we may assume that $m_k$ is always of the form
$\widehat{p}_0+(K+p)q$ for some $q$.  If $\widehat{p}_0>K$ then we define $p_0:=\widehat{p}_0-K$ otherwise let $p_0:=\widehat{p}_0+p$.

We will obtain $\nu^s_p$ as the weak limit of descending measures $\nu^s_{p,k}:\Sigma\to [0,1]$. The
construction is fairly intricate and
involves splitting the measure into blocks of length $p$ with ``buffers''  of length $K$ in-between that are given
by Condition \ref{thm:injection}. However, at each position $m_\ell$, we want to append
$\lambda_{m_\ell}$. To ensure consistency of lengths, we need to slightly modify $\lambda_{m_\ell}$
by extending the words to be of length $p+q(K+p)$ for some $q\geq0$. To this end we define
$\lambda_{m_\ell}' = \lambda_{m_\ell}11\dots1$, where the number of symbol $1$'s is $p-|\lambda_{m_\ell}|\mod(K+p)$. Let 
	\[
	\Omega(k) = \begin{cases}
		\{\lambda'_{m_\ell}\} &\text{if } k=m_\ell-1 \text{ for some }\ell\in\N,\\
		\Sigma_p & \text{otherwise.}
	\end{cases}
	\]
For every $\iii_1,\iii_2\in\Sigma_*$ denote the word in Condition \ref{thm:injection} by $\kkk(\iii_1,\iii_2)\in\Sigma_{K}$. We define a collection of symbols $\cK_n$ by induction. Let $\cK_0:=\Sigma_{p_0}$  Suppose that $\cK_n$ is defined for some $n\geq0$. Then let us define $\cK_{n+1}$ as 
$$
\cK_{n+1} = \{\iii\kkk\jjj : 
\iii\in\cK_n, \jjj\in \Omega(|\iii\kkk|)\text{ and }\kkk =\kkk(\iii,\jjj)\}.
$$
To ease notation let $\ell_k$ denote the length of words in $\cK_k$. Observe that by construction, every $\iii\in\cK_n$ can be written of the form
$$
\iii=\iii_1\kkk_1\iii_2\kkk_2\dots\kkk_n\iii_{n+1},
$$
where for every $k\in\{2,\ldots,n+1\}$, $\iii_k\in\Omega(\ell_{k-1}+K)$ and $\kkk_k=\kkk(\iii_1\kkk_1\dots\kkk_{k-1}\iii_{k},\iii_{k+1})$. While the cylinders in $\cK_n$ consist of the same number of blocks ($n+1$) and buffers ($n$), their
lengths are not necessarily $p_0+n(p+K)$ due to the different lengths of $\lambda'_{m_i}$. Their
lengths are however, by
construction, always of length $p_0+q(p+K)$ for some integer $q\geq n$. 
  This ensures that we can construct a $\limsup$ set of codings $\cK$ by
\[
  \cK = \bigcap_{k}\{[\iii] : \iii\in\cK_k\}
\]
whose image $\pi_{\bt}(\cK)$ is a non-empty subset
of the $\limsup$ set $R_{\bt}((\lambda_k)_k)$ by choice of $p_0$.

Let $\eta(n)$ denote the number of $\lambda_{m_\ell}'$ blocks in $\cK_n$. Then
\begin{equation}\label{eq:length}
\ell_n=\sum_{i=1}^{\eta(n)}|\lambda_{m_i}'|+(n-\eta(n))p+Kn+p_0.
\end{equation}

We start by defining $\nu_{p,0}^s$ on cylinders of length no less than $p_0$ by
\[
  \nu_{p,0}^s([\iii\hhh]) = \frac{\phi^s(\iii)}{\sum_{\jjj\in\Sigma_{p_0}}\phi^s(\jjj)}N^{-|\hhh|}
\]
for $\iii\in\Sigma_{p_0}=\cK_0$ and $\hhh\in\Sigma_*$. This uniquely defines a probability measure on $\Sigma$, i.e.\ $\nu_{p,0}^s(\Sigma)=1$.

We define $\nu_{p,n}^s$ on cylinders with prefix in $\cK_{n}$ by
\[
  \nu_{p,n+1}^s(\iii) = \begin{cases}
    \displaystyle\frac{\phi^s(\iii_1)\phi^s(\iii_2)\dots\phi^s(\iii_{n+1})}{\displaystyle\sum_{\iii_1\kkk_1\dots\kkk_n\iii_{n+1}\in\cK_{n}}\hspace{-2em}\phi^s(\iii_1)\phi^s(\iii_2)\dots\phi^s(\iii_{n+1})}N^{-|\iii|+|\iii_1\kkk_1\dots\kkk_n\iii_{n+1}|}& \text{if
    }\iii\prec\iii_1\kkk_1\dots\kkk_n\iii_{n+1}\in\cK_{n},\\
    0 & \text{otherwise.}
  \end{cases}
\]
Observe that for any cylinder set $O\subseteq\Sigma$, the measures $\nu_{p,k}^s(O)$ are eventually monotone decreasing and hence $\nu_p^s(O) \leq \nu_{p,k}^s(O)$ for all sufficiently large $k\in\N$, where $\nu_p^s$ is the weak limit of $(\nu_{p,k}^s)_{k\in\N}$.

\begin{lemma}
  \label{thm:comparingFeng}
  Let $k\in\N_0$. Then, 
  \begin{multline}
  \label{eq:maximalSumCorrespondence}
  C^{k}\sum_{\substack{\iii_1\in\Sigma_{p_0}\\\iii_{j}\in\Omega(\ell_{j-1}+K)\\j\leq k+1}}\phi^s(\iii_1)\phi^s(\iii_2)\dots\phi^s(\iii_{k+1})\\
  \leq
  \sum_{\iii_1\kkk_1\dots\kkk_k\iii_{k+1}\in\cK_{k}}\hspace{-1em}\phi^s(\iii_1\kkk_1\dots\kkk_k\iii_{k+1})\\
  \leq\sum_{\substack{\iii_1\in\Sigma_{p_0}\\\iii_{j}\in\Omega(\ell_{j-1}+K)\\j\leq k+1}}\phi^s(\iii_1)\phi^s(\iii_2)\dots\phi^s(\iii_{k+1}),
\end{multline}
where $C$ is the constant appearing in Condition~\ref{thm:injection}.
\end{lemma}
\begin{remark}
  Observe that the summations in \eqref{eq:maximalSumCorrespondence} are all over the same set. We have
  changed the subscript to emphasise these two points of view of $\cK_k$ versus its constituent
parts.
\end{remark}
\begin{proof}
  The last inequality follows from the submultiplicativity of $\phi^s$ and that
  $\phi^s(\kkk_j)<1$.

  The first inequality follows inductively from repeated application of 
  Condition \ref{thm:injection} as follows:
  The base case $k=0$ follows trivially, since $\cK_0 = \Sigma_{p_0}$.
  For the induction step assume that \eqref{eq:maximalSumCorrespondence} holds for $k\geq 0$.
  Applying Condition \ref{thm:injection} to words in $\cK_{k+1}$ gives
  \begin{align*}
    \sum_{\iii_1\kkk_1\dots\kkk_{k+1}\iii_{k+2}\in\cK_{k+1}}\hspace{-1em}\phi^s(\iii_1\kkk_1\dots\kkk_{k+1}\iii_{k+2})
    &\geq C
    \sum_{\substack{\iii_1\kkk_1\dots\kkk_{k}\iii_{k+1}\in\cK_{k}\\\iii_{k+2}\in\Omega(\ell_{k+1}+K)}}\hspace{-1em}
    \phi^s(\iii_1\kkk_1\dots\kkk_k\iii_{k+1})\phi^s(\iii_{k+2})
    \intertext{and the induction hypothesis immediately gives}
    &\geq C^{k+1}\sum_{\substack{\iii_1\in\Sigma_{p_0}\\\iii_{j}\in\Omega(\ell_{j-1}+K)\\j\leq
    k+2}}\phi^s(\iii_1)\phi^s(\iii_2)\dots\phi^s(\iii_{k+2}),
  \end{align*}
  which completes the proof.
\end{proof}

The proof of Theorem \ref{thm:mainTheorem} reduces mainly to the following technical lemma.
\begin{lemma}  \label{thm:finiteEnergy}
  Let $s_0>0$ be such that $P(s_0)=\alpha(s_0)$. 
  Then for all $0<t<s<s_0$ and sufficiently large $p$,
  \begin{equation}\label{eq:doubleInt}
  	\iint_{\Sigma\times\Sigma}\frac{d\nu_p^s(\iii)d\nu_p^s(\jjj)}{\phi^{t}(A_{\iii\wedge\jjj})} < \infty. 
  \end{equation}
\end{lemma}
\begin{proof}
	Let $p\in\N$ be large enough such that $\gamma^{(s-t)p}<C$, where $0<\gamma<1$ and $1>C>0$ are the
	constants appearing in Lemma \ref{thm:singularratio} and Condition \ref{thm:injection}, respectively. Since $s<s_0$, we have $P(s)>\alpha(s)$ and we can pick $\delta>0$ such that $P(s)-\alpha(s)>
	4\delta$ and choose $p$ (which so far only depends on the $C$ and $\gamma$) large enough such that we may apply
	\eqref{eq:targetEstimate} and \eqref{eq:pressureEstimate} with $\delta$, moreover, we
	require that $p\delta>KP(s)-2K\delta$.
	
  Recall that $\nu_p^s$ is supported on $\cK$ and note that for all distinct $\iii,\jjj\in\cK$, their longest common prefix $\iii\wedge\jjj$
  must be a word of the form
\(
  \iii_1\kkk_1\dots \iii_n\iii'
\)
for some $\iii'\in\Sigma_{\leq (p+K)}=\bigcup_{k=0}^{p+K}\Sigma_k$ and $n$ maximal. To see this, assume $|\iii'|> p+K$. Then, $\iii'$ must have a
prefix of the form $\kkk_n\lambda_{m_j}'$. But since all words $\hhh\in\cK$ satisfy
$(\sigma^{m_j}\hhh)|_{|\lambda_{m_j}'|}=\lambda_{m_j}'$, so must $\jjj$ and we obtain
\[
  \iii\wedge\jjj = \iii_1\kkk_1 \dots\iii_n \kkk_n \lambda_{m_j}' \iii''
\]
for some finite word $\iii''$. This however, contradicts the maximality of $n$ and our claim
follows.

Note further that by the boundedness of the length of $\iii'$ by $p+K$ as well as the
non-singularity of the matrices $A_i$, there exists a universal constant $D$ for the IFS
such that 
  \begin{equation}
    \label{eq:shortBounded}
    1/D^{t(p+2K)} \phi^{t}(\jjj) \leq \phi^{t}(\jjj\iii') \leq D^{t(p+2K)} \phi^{t}(\jjj).
  \end{equation}

  The double integral \eqref{eq:doubleInt}, together with \eqref{eq:shortBounded} simplifies to the following sum
  \[
    \sum_{n=0}^{\infty}\;\sum_{\substack{\iii\wedge\jjj=\iii_1\dots\iii_n\iii'\\\iii,\jjj\in\cK,
    \,\iii'\in\Sigma_{\leq
p+K}}}\frac{\nu_p^s([\iii|_{|\iii\wedge\jjj|}])\nu_p^s([\jjj|_{|\iii\wedge\jjj|}])}{\phi^{t}(\iii\wedge\jjj)}
\leq
D^{t(p+2K)}N^{p+K}\sum_{n=0}^{\infty}\sum_{\substack{\iii\in\cK_{n}}}\frac{\nu_p^s([\iii])^2}{\phi^{t}(\iii)}.
  \]
  Thus,
  \begin{align*}
    &\iint_{\Sigma\times\Sigma}\frac{d\nu_p^s(\iii)d\nu_p^s(\jjj)}{\phi^{t}(A_{\iii\wedge\jjj})}\\
    &\leq D^{t(p+2K)} N^{p+K}\sum_{n=0}^{\infty}
    \sum_{\iii\in\cK_n}\frac{\nu_{p,n}^s([\iii_1\kkk_1\dots\kkk_n\iii_{n+1}])^2}{\phi^{t}(\iii_1\kkk_1\dots\kkk_n\iii_{n+1})}\\
    &=D^{t(p+2K)} N^{p+K} \sum_{n=0}^{\infty} \sum_{\iii\in\cK_n} \left(
      \frac{\phi^s(\iii_1)\phi^s(\iii_2)\dots\phi^s(\iii_{n+1})}{\sum_{\jjj\in\cK_n}\phi^s(\jjj_1)\phi^s(\jjj_2)\dots\phi^s(\jjj_{n+1})}
	\right)^2
	\phi^{t}(\iii_1\kkk_1\dots\kkk_n\iii_{n+1})^{-1}
	\intertext{and by definition of $\cK_n$,}
	&= D^{t(p+2K)}N^{p+K} \sum_{n=0}^{\infty}\left(
      \frac{1}{\sum_{\jjj\in\cK_n}\phi^s(\jjj_1)\phi^s(\jjj_2)\dots\phi^s(\jjj_{n+1})}
    \right)^2\\
    &\hspace{26em}\cdot\sum_{\iii\in\cK_n}\frac{(\phi^s(\iii_1)\phi^s(\iii_2)\dots\phi^s(\iii_{n+1}))^2}{\phi^{t}(\iii_1\kkk_1\dots\kkk_n\iii_{n+1})}\\
    &\leq D^{t(p+2K)} N^{p+K} \sum_{n=0}^{\infty}\left(
      \frac{1}{\sum_{\jjj\in\cK_n}\phi^s(\jjj_1)\phi^s(\jjj_2)\dots\phi^s(\jjj_{n+1})}
    \right)^2\\
    &\hspace{15em}\cdot\sum_{\iii\in\cK_n}C^{-n}\cdot\frac{\phi^s(\iii_1)\phi^s(\iii_2)\dots\phi^s(\iii_{n+1})\phi^s(\iii_1\kkk_1\dots\kkk_n\iii_{n+1})}{\phi^{t}(\iii_1\kkk_1\dots\kkk_n\iii_{n+1})}\\
    &\leq D^{t(p+2K)} N^{p+K} \sum_{n=0}^{\infty}C^{-n}\gamma^{(s-t)\ell_n}\left( \sum_{\jjj\in\cK_n}\phi^s(\jjj_1)\phi^s(\jjj_2)\dots\phi^s(\jjj_{n+1})
    \right)^{-1}
\end{align*}
by Lemma \ref{thm:singularratio} for some $0<\gamma<1$. Again, let 
$\eta(n)$ denote the number of $\lambda_{m_l}$ blocks in $\cK_n$.
Using Lemma \ref{thm:comparingFeng}, we can bound
\begin{align*}
  &= D^{t(p+2K)} N^{p+K} \sum_{n=0}^{\infty}C^{-n}\gamma^{(s-t)\ell_n}\left( \sum_{\substack{\jjj_1\in\Sigma_{p_0}\\\jjj_{k}\in\Omega(\ell_{k-1}+K)\\k\leq n+1}}\phi^s(\jjj_1)\phi^s(\jjj_2)\dots\phi^s(\jjj_{n+1})
    \right)^{-1}\\
    &\leq c D^{t(p+2K)} N^{p+K}\sum_{n=1}^{\infty}C^{-n}\gamma^{(s-t)\ell_n}\left(
      \left( \sum_{\jjj\in\Sigma_p}\phi^s(\jjj) \right)^{n-\eta(n)}
      \cdot\prod_{i=1}^{\eta(n)}\phi^s(\lambda_{m_i}')
    \right)^{-1}
\intertext{for some $c>0$. Then by \eqref{eq:targetEstimate} and \eqref{eq:pressureEstimate} }
  &\leq c D^{t(p+2K)} N^{p+K}\sum_{n=1}^{\infty}C^{-n}\gamma^{(s-t)\ell_n}
  e^{-(n-\eta(n))p(P(s)-\delta)}\prod_{i=1}^{\eta(n)}e^{m_i(\alpha(s)+\delta)}.
  \intertext{Applying \eqref{eq:length} and $p\delta>KP(s)-2K\delta \Leftrightarrow
  \tfrac{p}{p+K}(P(s)-\delta)>P(s)-2\delta$ we get}
  &=c D^{t(p+2K)} N^{p+K}\sum_{n=1}^{\infty}C^{-n}\gamma^{(s-t)\ell_n}
  \exp\Bigg(-\left(\ell_n-\sum_{i=1}^{\eta(n)}|\lambda_{m_i}'|-K\eta(n)-p_0\right)\\
  &\hspace{25em}\cdot\frac{p}{p+K}(P(s)-\delta)+(\alpha(s)+\delta)\sum_{i=1}^{\eta(n)}m_i\Bigg)\\
  &\leq c' D^{t(p+2K)} N^{p+K}\sum_{n=1}^{\infty}C^{-n}\gamma^{(s-t)\ell_n}
  \exp\left(-\left(\ell_n-\sum_{i=1}^{\eta(n)}|\lambda_{m_i}'|-K\eta(n)\right)(P(s)-2\delta)+(\alpha(s)+\delta)\sum_{i=1}^{\eta(n)}m_i\right).
  \intertext{Clearly, $\ell_n\geq m_{\eta(n)}+|\lambda_{m_{\eta(n)}}'|$ and $\ell_n\geq\sum_{i=1}^{\eta(n)}|\lambda_{m_i}'|+ (n-\eta(n))p\geq np$ so}
&\leq c' D^{t(p+2K)} N^{p+K}\sum_{n=1}^{\infty}C^{-n}\gamma^{(s-t)pn}
\\&\hspace{12em}\cdot  \exp\left(-\left(m_{\eta(n)}-\sum_{i=1}^{\eta(n)-1}|\lambda_{m_i}'|-K\eta(n)\right)(P(s)-2\delta)+(\alpha(s)+\delta)\sum_{i=1}^{\eta(n)}m_i\right).
\intertext{Now we can apply \eqref{eq:growthcondition} to obtain,}
&\leq c'' D^{t(p+2K)} N^{p+K}\sum_{n=1}^{\infty}C^{-n}\gamma^{(s-t)pn}
\\&\hspace{12em}\cdot
\exp\left(-(1-2^{-n})m_{\eta(n)}(P(s)-2\delta)+(1+2^{-n})m_{\eta(n)}(\alpha(s)+\delta)\right)\\
&\leq c'' D^{t(p+2K)} N^{p+K}\sum_{n=1}^{\infty}C^{-n}\gamma^{(s-t)pn}
\exp\left(m_{\eta(n)}((1+2^{-n})\alpha(s)-(1-2^{-n})P(s)+3\delta)\right).
\end{align*}
Coupling this with the observation that $C^{-1}\gamma^{(s-t)p} <1$ and
$(1+2^{-n})\alpha(s)-(1-2^{-n})P(s))+3\delta<0$ for sufficiently large $n$, the expression above is
bounded by a geometric series with ratio less than one and hence is bounded.
It immediately follows that \eqref{eq:doubleInt} is
bounded and the $t$ energy of $\nu_p^s$ is finite, as required.
\end{proof}

\begin{proof}[Proof of Theorem~\ref{thm:mainTheorem}]
	To show that $\dim_H R_{\bt}((\lambda_k)_k)\geq s_0$ for Lebesgue-almost every $\bt$, it is enough to show that for every $t<s_0$ we have $\dim_H R_{\bt}((\lambda_k)_k)\geq t$ for Lebesgue-almost every $\bt$. 
	
	Let $t<s<s_0$ and $p$ be as in Lemma~\ref{thm:finiteEnergy}. By Frostman's lemma (see for example \cite[Chapter~8]{Mattila}), it is enough to show that
	$$
	\iiint\frac{d(\pi_{\bt})_*\nu_p^s(x)d(\pi_{\bt})_*\nu_p^s(y)}{|x-y|^t}d\bt<\infty.
	$$
	By \cite[Lemma~3.1]{Falconer88} and \cite[Proposition~3.1]{Solomyak98}, there exists a constant $C>0$ such that
	$$
	\iiint\frac{d(\pi_{\bt})_*\nu_p^s(x)d(\pi_{\bt})_*\nu_p^s(y)}{|x-y|^t}d\bt\leq C\iint\frac{d\nu_p^s(\iii)d\nu_p^s(\jjj)}{\varphi^t(A_{\iii\wedge\jjj})},
	$$
	where the right-hand side is finite by Lemma~\ref{thm:finiteEnergy}.
	
Now, let us turn to the proof that $\mathcal{L}_d(R_{\bt}((\lambda_k)_k))>0$ for Lebesgue-almost every $\bt$ if $s_0>d$. Let $s$ be such that $s_0>s>d$. It is enough to show that $(\pi_{\bt})_*\nu_p^s\ll\mathcal{L}_d$, and by \cite[Theorem~2.12]{Mattila}, to do so it is enough to prove that
$$
\iint\liminf_{r\to0}\frac{(\pi_{\bt})_*\nu_p^s(B(y,r))}{r^d}d(\pi_{\bt})_*\nu_p^s(y)d\bt<\infty.
$$
By \cite[Proof of Proposition~2]{JPS} and \cite[Proposition~3.1]{Solomyak98}, there exists $C>0$ such that
$$
\iint\liminf_{r\to0}\frac{(\pi_{\bt})_*\nu_p^s(B(y,r))}{r^d}d(\pi_{\bt})_*\nu_p^s(y)d\bt\leq C\iint\frac{d\nu_p^s(\iii)d\nu_p^s(\jjj)}{\varphi^d(A_{\iii\wedge\jjj})},
$$
where the right-hand side is finite again by Lemma~\ref{thm:finiteEnergy}.
\end{proof}

\section{Pointwise recurrent sets}\label{sec:proofr}

\subsection{Proof of the upper bound}\label{sec:ubrec}

Let $\psi\colon\N\mapsto\N$ and let $\beta=\liminf_{n\to\infty}\frac{\psi(n)}{n}$. Suppose that $\beta<1$. For a finite word $\iii\in\Sigma_*$, let $\overline{\iii}\in\Sigma$ be the infinite word $\overline{\iii}=\iii\iii\iii\cdots$. Note that $S_{\bt}(\psi)$ can be written as 
$$
S_{\bt}(\psi)=\bigcap_{k_0=1}^\infty\bigcup_{k=k_0}^\infty\bigcup_{\iii\in\Sigma_k}\pi_{\bt}([\overline{\iii}|_{k+\psi(k)}]).
$$
Analogous to the properties of $\alpha(t)$, the map $t\mapsto (1-\beta)P(t)-\beta P_2(t)$ is
continuous, strictly decreasing, converges to infinity and $(1-\beta)P(0)-\beta P_2(0)=\log N$. Thus,
there exists a unique solution $r_0>0$ of the equation $(1-\beta)P(r_0)=P_2(r_0)$. Let $t>r_0$ be
arbitrary.

Since $0<(1-\beta)P(t)<\beta P_2(t)$, one can choose $\delta>0$ sufficiently small such that 
$$
P(t)+P_2(t)-2\delta>0\text{ and }(1-\beta+\delta)P(t)-(\beta-\delta)(P_2(t)-2\delta)+\delta<0.
$$ 
There exists a constant $C>0$ such that for every $k\geq1$
$$
\sum_{\iii\in\Sigma_k}\varphi^t(\iii)\leq Ce^{k(P(t)+\delta)}\text{ and }\sum_{\iii\in\Sigma_k}\left(\varphi^t(\iii)\right)^2\leq Ce^{k(-P_2(t)+\delta)}.
$$
Moreover, we may assume that $C>0$ is sufficiently large such that $k(\beta-\delta)-C\leq\psi(k)$
for every $k\geq1$. Hence, writing $\psi(k)_k:=\psi(k)\mod k$, we obtain
\[
\begin{split}
\mathcal{H}^t(S_{\bt}(\psi))&\leq\liminf_{k_0\to\infty}\sum_{k=k_0}^\infty\sum_{\iii\in\Sigma_k}\varphi^t(\overline{\iii}|_{k+\psi(k)})\\
&\leq\liminf_{k_0\to\infty}\sum_{k=k_0}^\infty\sum_{\iii\in\Sigma_{\psi(k)_k}}\sum_{\jjj\in\Sigma_{k-\psi(k)_k}}\left(\varphi^t(\iii)\right)^{\lfloor\frac{k+\psi(k)}{k}\rfloor+1}\left(\varphi^t(\jjj)\right)^{\lfloor\frac{k+\psi(k)}{k}\rfloor}\\
&\leq\liminf_{k_0\to\infty}\sum_{k=k_0}^\infty\sum_{\iii\in\Sigma_{\psi(k)_k}}\sum_{\jjj\in\Sigma_{k-\psi(k)_k}}\left(\varphi^t(\iii)\right)^{2}\varphi^t(\jjj)\\
&\leq C^2\liminf_{k_0\to\infty}\sum_{k=k_0}^\infty e^{k(P(t)+\delta)-\psi(k)(P(t)+P_2(t)-2\delta)}\\
&\leq C^2e^{C(P(t)+P_2(t)-2\delta)}\liminf_{k_0\to\infty}\sum_{k=k_0}^\infty e^{k\left(P(t)+\delta-(\beta-\delta)(P(t)+P_2(t)-2\delta)\right)}=0.
\end{split}
\]
Since $t>r_0$ was arbitrary, the claim follows.

\begin{remark}
  We note that if $\beta>1$ then the argument above is not optimal.
\end{remark}

\subsection{Lower bound for Theorem~\ref{thm:mainTheorem2}}\label{sec:lbrec}

The proof is analogous to the lower bound of Theorem~\ref{thm:mainTheorem} with some necessary modifications. Let $p\in\N$ an integer which will be specified later. Let $m_k$ be a sequence on which the lower limit $\beta=\liminf_{n\to\infty}\psi(n)/n$ is achieved and take a sparse subsequence such that
\begin{equation}\label{eq:constr}
\sum_{k=1}^nm_k\leq(1+2^{-n})m_n\text{ and }m_n\geq2^n\sum_{k=1}^{n-1}(\psi(m_k)+K).
\end{equation}
Let us choose $p_0$ as in Section~\ref{sec:lowerBound}, so $m_k=p_0+(p+K)q$ for every $k\geq1$ for some $q\in\N$. To ensure consistency of lengths again, we need to slightly modify $\psi(m_\ell)$
by extending the words to be of length $p+q(K+p)$ for some $q\geq0$. To this end we define
$\psi'(\ell):= \psi(m_\ell)+k$, where $k=p-\psi(m_\ell)\mod(K+p)$.

We construct a measure $\nu^s_p$ similarly to Section~\ref{sec:lowerBound}, except that the elements in $\Omega(k)$ depend on the previous elements. More precisely, let 
\[
\Omega(\iii,k) = \begin{cases}
	\{\iii|_{\psi'(\ell)}\} &\text{if } k=m_\ell-1 \text{ for some }\ell\in\N,\\
	\Sigma_p & \text{otherwise.}
\end{cases}
\]
For every $\iii_1,\iii_2\in\Sigma_*$ denote the word in Condition \ref{thm:injection} by $\kkk(\iii_1,\iii_2)\in\Sigma_{K}$. We define a collection of symbols $\cK_n$ by induction. Let $\cK_0':=\Sigma_{p_0}$  Suppose that $\cK_n'$ is defined for some $n\geq0$. Then let us define $\cK_{n+1}'$ as 
$$
\cK_{n+1}' = \{\iii\kkk\jjj : 
\iii\in\cK_n', \jjj\in \Omega(\iii,|\iii\kkk|)\text{ and }\kkk =\kkk(\iii,\jjj)\}.
$$
Denote by $\ell_k'$ the length of words in $\cK_k'$. Observe that by construction, again every $\iii\in\cK_n$ can be written of the form
$$
\iii=\iii_1\kkk_1\iii_2\kkk_2\dots\kkk_n\iii_{n+1},
$$
where for every $k\in\{2,\ldots,n+1\}$, $\iii_k\in\Omega(\iii_1\kkk_1\dots\iii_{k-1},\ell_{k-1}+K)$ and $\kkk_k=\kkk(\iii_1\kkk_1\dots\kkk_{k-1}\iii_{k},\iii_{k+1})$.

Let $\eta'(n)$ denote the number of recurrences in $\cK_n$. Then
\begin{equation}\label{eq:length'}
	\ell_n=\sum_{i=1}^{\eta(n)}\psi'(i)+(n-\eta(n))p+Kn+p_0.
\end{equation}

We start by defining $\nu_{p,0}^s$ on cylinders of length no less than $p_0$ by
\[
\nu_{p,0}^s([\iii\hhh]) = \frac{\phi^s(\iii)}{\sum_{\jjj\in\Sigma_{p_0}}\phi^s(\jjj)}N^{-|\hhh|}
\]
for $\iii\in\Sigma_{p_0}=\cK_0$ and $\hhh\in\Sigma_*$. This uniquely defines a probability measure on $\Sigma$, i.e.\ $\nu_{p,0}^s(\Sigma)=1$.

We define $\nu_{p,n}^s$ on cylinders with prefix in $\cK_{n}$ by
\[
\nu_{p,n+1}^s(\iii) = \begin{cases}
	\displaystyle\frac{\phi^s(\iii_1)\phi^s(\iii_2)\dots\phi^s(\iii_{n+1})}{\displaystyle\sum_{\iii_1\kkk_1\dots\kkk_n\iii_{n+1}\in\cK_{n}}\hspace{-2em}\phi^s(\iii_1)\phi^s(\iii_2)\dots\phi^s(\iii_{n+1})}N^{-|\iii|+|\iii_1\kkk_1\dots\kkk_n\iii_{n+1}|}& \text{if
	}\iii\prec\iii_1\kkk_1\dots\kkk_n\iii_{n+1}\in\cK_{n},\\
	0 & \text{otherwise.}
\end{cases}
\]
Observe that for any open subset of $O\subseteq\Sigma$, the measures $\nu_{p,k}^s(O)$ are
monotone decreasing and hence $\nu_p^s(O) \leq \nu_{p,k}^s(O)$ for all $k\in\N$, where $\nu_p^s$ is
the weak limit of $(\nu_{p,k}^s)_{k\in\N}$.

\begin{lemma}\label{thm:finiteEnergy2}
	Let $r_0>0$ be such that $(1-\beta)P(r_0)=\beta P_2(r_0)$. 
	Then for all $0<t<s<r_0$ and sufficiently large $p$,
	\begin{equation*}
		\iint_{\Sigma\times\Sigma}\frac{d\nu_p^s(\iii)d\nu_p^s(\jjj)}{\phi^{t}(A_{\iii\wedge\jjj})} < \infty. 
	\end{equation*}
\end{lemma}
\begin{proof}
	By similar argument to the beginning of Lemma~\ref{thm:finiteEnergy}, it is enough to show that
	$$
	\mathcal{I}=\sum_{n=0}^\infty\sum_{\iii\in\cK_n'}\frac{\nu_p^s([\iii])^2}{\varphi^t(\iii)}<\infty.
	$$
	 \begin{align*}
		\mathcal{I}&=\sum_{n=0}^{\infty}
		\sum_{\iii\in\cK_n'}\frac{\nu_{p,n}^s([\iii_1\kkk_1\dots\kkk_n\iii_{n+1}])^2}{\phi^{t}(\iii_1\kkk_1\dots\kkk_n\iii_{n+1})}\\
		&=\sum_{n=0}^{\infty} \sum_{\iii\in\cK_n} \left(
		\frac{\phi^s(\iii_1)\phi^s(\iii_2)\dots\phi^s(\iii_{n+1})}{\sum_{\jjj\in\cK_n}\phi^s(\jjj_1)\phi^s(\jjj_2)\dots\phi^s(\jjj_{n+1})}
		\right)^2
		\phi^{t}(\iii_1\kkk_1\dots\kkk_n\iii_{n+1})^{-1}\\
		&=\sum_{n=0}^{\infty}\left(
		\frac{1}{\sum_{\jjj\in\cK_n'}\phi^s(\jjj_1)\phi^s(\jjj_2)\dots\phi^s(\jjj_{n+1})}
		\right)^2\cdot\sum_{\iii\in\cK_n}\frac{(\phi^s(\iii_1)\phi^s(\iii_2)\dots\phi^s(\iii_{n+1}))^2}{\phi^{t}(\iii_1\kkk_1\dots\kkk_n\iii_{n+1})}\\
		&\leq \sum_{n=0}^{\infty}C^{-n}\gamma^{(s-t)\ell_n}\left( \sum_{\jjj\in\cK_n'}\phi^s(\jjj_1)\phi^s(\jjj_2)\dots\phi^s(\jjj_{n+1})
		\right)^{-1}.
	\end{align*}
Denote by $\eta(n)$ the number of returns in $\cK_n'$. By definition, $m_{\eta(n)}$ is the position of the last return, and it returns to
$[\jjj|_{\psi(m_{\eta(n)})}]$. Unfortunately, $\jjj|_{\psi(m_{\eta(n)})}$ is not necessarily an
element of $\cK_k'$ for all $k>0$. Let $k_n$ be the smallest integer such that
$\psi(m_{\eta(n)})\leq \ell_{k_n}$, where we recall that $\ell_n$ is the length of the elements of
$\cK_n'$. Clearly, for every $\jjj=\jjj_1\kkk_1\dots\kkk_{k_n}\jjj_{k_n+1}\in\cK_{k_n}'$
$$
\phi^s(\jjj_1)\phi^s(\jjj_2)\dots\phi^s(\jjj_{k_n+1})\geq\phi^s(\jjj),
$$
and for $\jjj\in\cK_n'$ $\phi^s(\jjj|_{\psi(m_{\eta(n)})}))\geq\phi^s(\jjj')$, where $\jjj'$ is the unique element in $\cK_{k_n}'$ such that $\jjj\prec\jjj'$. Moreover, for every $\jjj\in\cK_n'$ there are $n-\eta(n)-(k_n-\eta(k_n))$-many $\Sigma_p$ components in the sequence $\sigma^{\ell_{k_n}}\jjj$. Hence, we obtain that
\begin{align*}
	&\leq\sum_{n=1}^{\infty}C^{-n}\gamma^{(s-t)\ell_n}\left(
	\left( \sum_{\jjj\in\Sigma_p}\phi^s(\jjj) \right)^{n-\eta(n)-(k_n-\eta(k_n))}\cdot \left( \sum_{\jjj\in\cK_{k_n}'}\phi^s(\jjj)^2 \right)
	\cdot\underline{\gamma}^{s\sum_{i=1}^{\eta(n)-1}\psi(m_i)}
	\right)^{-1}
	\end{align*}
	\begin{align*}
	&\leq \sum_{n=1}^{\infty}C^{-n}\gamma^{(s-t)\ell_n}\left(
	\Bigg( \sum_{\jjj\in\Sigma_p}\phi^s(\jjj) \right)^{n-\eta(n)-(k_n-\eta(k_n))}\cdot \left( \sum_{\jjj\in\Sigma_p}\phi^s(\jjj)^2 \right)^{k_n-\eta(k_n)}\\
	&\hspace{28em}\cdot\underline{\gamma}^{s(\sum_{i=1}^{\eta(n)-1}\psi(m_i)+2\sum_{i=1}^{\eta(k_n)}\psi(m_i))}
	\Bigg)^{-1}\\
	&\leq \sum_{n=1}^{\infty}C^{-n}\gamma^{(s-t)\ell_n}\exp\Bigg(
	-p(P(s)-\delta)(n-\eta(n)-(k_n-\eta(k_n))+p(P_2(s)+\delta)(k_n-\eta(k_n))\\
	&\hspace{25em}-\log\underline{\gamma}s\left(\sum_{i=1}^{\eta(n)-1}\psi(m_i)+2\sum_{i=1}^{\eta(k_n)}\psi(m_i)\right)
	\Bigg)
      \end{align*}
      \begin{align*}
	&\leq \sum_{n=1}^{\infty}C^{-n}\gamma^{(s-t)\ell_n}\exp\Bigg(
	-p(P(s)-\delta)(n-\eta(n))+p(P(s)+P_2(s))(k_n-\eta(k_n))\\
	&\hspace{25em}-\log\underline{\gamma}s\left(\sum_{i=1}^{\eta(n)-1}\psi(m_i)+2\sum_{i=1}^{\eta(k_n)}\psi(m_i)\right)
	\Bigg).
	\intertext{ Using \eqref{eq:length'}, we get}
	&\leq \sum_{n=1}^{\infty}C^{-n}\gamma^{(s-t)\ell_n}\exp\Bigg(
	-(P(s)-\delta)(\ell_n-\sum_{i=1}^{\eta(n)}\psi(m_i)-K\eta(n)-p_0)\\
	&\hspace{05em}+(P(s)+P_2(s))(\ell_{k_n-1}+p(\eta(k_n-1)-\eta(k_n)))-\log\underline{\gamma}s\left(\sum_{i=1}^{\eta(n)-1}\psi(m_i)+2\sum_{i=1}^{\eta(k_n)}\psi(m_i)\right)
	\Bigg)
      \end{align*}
      \begin{align*}
	&\leq \sum_{n=1}^{\infty}C^{-n}\gamma^{(s-t)\ell_n}\exp\Bigg(
	-(P(s)-\delta)(m_{\eta(n)}-\sum_{i=1}^{\eta(n)-1}\psi(m_i)-K\eta(n)-p_0)\\
	&\hspace{15em}+(P(s)+P_2(s))(\psi(m_{\eta(n)})+p)-3\log\underline{\gamma}s\left(\sum_{i=1}^{\eta(n)-1}\psi(m_i)\right)
	\Bigg)
	\intertext{ Using the defining properties \eqref{eq:constr} of the sequence $m_n$, we have }
	&\leq c\sum_{n=1}^{\infty}C^{-n}\gamma^{(s-t)\ell_n}\exp\Bigg(
	-(P(s)-\delta)(1-2^{-n})m_{\eta(n)}+(P(s)+P_2(s))(\beta+\delta)m_{\eta(n)}-3\log\underline{\gamma}s2^{-n}m_{\eta(n)}
	\Bigg)
\end{align*}
Coupling this with the observation that $C^{-1}\gamma^{(s-t)p} <1$ and
$-(P(s)-\delta)(1-2^{-n})+(P(s)+P_2(s))(\beta+\delta)-2\log\underline{\gamma}s2^{-n}<0$ for
sufficiently large $n$, the left hand side is finite and the proof is complete.
\end{proof}

Now, the proof of Theorem~\ref{thm:mainTheorem2} is identical to the proof of Theorem~\ref{thm:mainTheorem} by replacing Lemma~\ref{thm:finiteEnergy} with Lemma~\ref{thm:finiteEnergy2}, so we omit it.

\section{Justification of Condition~\ref{thm:injection}}\label{sec:condition}
In this section, we give a sufficient condition under which Condition~\ref{thm:injection} holds. The proof is not only a modification of the proof but also an application of K\"aenm\"aki and Morris \cite[Proposition~4.1]{Kaenmaki18}. First, let us recall some definitions and notations from algebraic geometry, following Goldsheid and Guivarc'h \cite{GolGui} and K\"aenm\"aki and Morris \cite{Kaenmaki18}. 

Let us denote by $\wedge^k\R^d$ the \textit{$k$th exterior product of $\R^d$}. That is, let $\{e_1,\ldots,e_d\}$ be the standard orthonormal basis of $\R^d$ and define
$$
\wedge^k\R^d=\mathrm{span}\{e_{i_1}\wedge\cdots\wedge e_{i_k}:1\leq i_1<\cdots<i_k\leq d\}.
$$
for all $k=1,\ldots,d$ and let $\wedge^0\R^d=\R$ by convention. The \textit{wedge product $\wedge\colon\wedge^k\R^d\times\wedge^j\R^d\mapsto\wedge^{k+j}\R^d$} is an associative bilinear operator, which is anticommutative on the elements or $\R^d$, i.e. for $w\in\wedge^k\R^d$ and $v\in\wedge^{j}\R^d$
$$
w\wedge v=(-1)^{kj}v\wedge w.
$$
If $v\in\wedge^k\R^d$ can be expressed as a wedge product of $k$ vectors of $\R^d$ then $v$ is said
to be \textit{decomposable}. Let us define the \textit{Hodge star operator}
$*\colon\wedge^k\R^d\mapsto\wedge^{d-k}\R^d$ to be the bijective linear map satisfying 
$$
*(e_{i_1}\wedge\cdots\wedge e_{i_k})=\mathrm{sgn}(i_1,\ldots,i_d)e_{i_{k+1}}\wedge\cdots\wedge e_{i_d}
$$
for all $1\leq i_1<\cdots<i_k\leq d$, where $1\leq i_{k+1}<\cdots<i_d\leq d$ are such that $\{i_{k+1},\ldots,i_d\}=\{1,\ldots,d\}\setminus\{i_1,\ldots,i_k\}$ and $\mathrm{sgn}(i_1,\ldots,i_d)$ is the signature of the permutation $(i_1,\ldots,i_d)$ of $(1,\ldots,d)$. Let us define the \textit{inner product on $\wedge^k\R^d$} by
$$
\langle v,w\rangle_k=*(v\wedge(*w))
$$
for all $v,w\in\wedge^k\R^d$. Moreover, we define the norm of $v\in\wedge^k\R^d$ by $\|v\|_k=\sqrt{\langle v,v\rangle_k}$. It can be shown that if $v,w\in\wedge^k\R^d$ are decomposable elements then 
$$
\langle v,w\rangle_k=\det(\langle v_i,w_j\rangle),
$$
where $v=v_1\wedge\cdots\wedge v_k$ and $w=w_1\wedge\cdots\wedge w_k$. For $A\in \GL_d(\R)$, we can define an invertible linear map $A^{\wedge k}\colon\wedge^k\R^d\mapsto\wedge^k\R^d$ by setting
$$
A^{\wedge k}(e_{i_1}\wedge\cdots\wedge e_{i_k})=(Ae_{i_1})\wedge\cdots\wedge(Ae_{i_k})
$$
and extending by linearity. 

For every matrix $A\in \GL_d(\R)$, there exists a basis of orthonormal vectors $\{u_1,\ldots,u_d\}$ such that $\|Au_i\|=\alpha_i(A)$ and $\{\alpha_1(A)^{-1}Au_1,\ldots,\alpha_d(A)^{-1}Au_d\}$ is orthonormal. Hence, the operator norm of $A^{\wedge k}$ is
$$
\|A^{\wedge k}\|_k=\max\{\|A^{\wedge k}w\|_k:\|w\|_k=1\}=\|A^{\wedge k}(u_1\wedge\cdots\wedge u_k)\|_k=\alpha_1(A)\cdots\alpha_k(A).
$$
Thus, for every $0<s\leq d$, the singular value function can be written as
$$
\varphi^s(A)=\left(\|A^{\wedge \lfloor s\rfloor}\|_{\lfloor s\rfloor}\right)^{1+\lfloor s\rfloor -s}\left(\|A^{\wedge\lceil s\rceil}\|_{\lceil s\rceil}\right)^{s-\lfloor s\rfloor}.
$$

Let $\bA = \{A_1, A_2, \cdots, A_N\}$ be a tuple of $\GL_d(\R)^N$ matrices. We say that $\bA$ is \textit{$k$-irreducible} if there is no proper subspace $V$ of $\wedge^k\R^d$ such that $A^{\wedge k}V=V$ for every $A\in\bA$. Similarly, we say that $\bA$ is \textit{strongly $k$-irreducible} if there is no finite collection of proper subspaces $V_1,\ldots,V_n$ of $\wedge^k\R^d$ such that $\bigcup_{k=1}^n\bigcup_{A\in\bA}A^{\wedge k}V_k=\bigcup_{k=1}^nV_k$. Denote by $\mathcal{S}(\bA)$ the semi-group induced by $\bA$. The following lemma is due to K\"aenm\"aki and Morris \cite[Proposition~4.1]{Kaenmaki18}.

\begin{lemma}\label{lem:kaenmor}
	Let $\bA$ be a tuple of matrices of $\GL_d(\R)$ such that $\bA$ is $k$- and $k+1$-irreducible. If there exist nonzero vectors $v_k,w_k\in\wedge^k\R^d$ and $v_{k+1},w_{k+1}\in\wedge^{k+1}\R^d$ such that
	$$
	\langle v_k,A^{\wedge k}w_k\rangle_k\langle v_{k+1},A^{\wedge {k+1}}w_{k+1}\rangle_{k+1}=0
	$$
	for every $A\in\mathcal{S}(\bA)$ then $\bA$ is neither strongly $k$-irreducible nor strongly $(k+1)$-irreducible.
\end{lemma}

For two vectorspaces $V$ and $W$, let us define the \textit{tensor product} $V\otimes W$ as follows
$$
V\otimes W=\mathrm{span}\{v\otimes w:v\in V,\ w\in W\},
$$
where for any $v_1,v_2\in V$, $w_1,w_2\in W$ and $\alpha\in\R$ 
\[
\begin{split}
&(v_1+v_2)\otimes w_1=v_1\otimes w_1+v_2\otimes w_1,\\
&v_1\otimes(w_1+w_2)=v_1\otimes w_1+v_1\otimes w_2,\\
&\alpha(v_1\otimes w_1)=(\alpha v_1)\otimes w_1=v_1\otimes(\alpha w_1).
\end{split}
\]
Let us consider the following tensor product of the exterior algebras
$$
\widehat{W}=\wedge^1\R^d\otimes\cdots\otimes\wedge^{d-1}\R^d.
$$
We define the inner product of $\widehat{W}$ for $u=u_1\otimes\cdots\otimes u_{d-1}, v=v_1\otimes\cdots\otimes v_{d-1}\in\widehat{W}$
$$
\langle u,v\rangle_\wedge=\prod_{i=1}^{d-1}\langle u_i,v_i\rangle_i,
$$
and extend it in a bilinear, symmetric way. We define a linear subspace $W$ of $\widehat{W}$, which is generated by the flags of $\R^d$ as follows:
$$
W=\mathrm{span}\{u_1\otimes(u_1\wedge u_2)\otimes\cdots\otimes(u_1\wedge\cdots\wedge u_{d-1}):\{u_1,\ldots,u_{d-1}\}\text{ linearly independent in }\R^d\}.
$$
We call $W$ the \textit{flag vector space}. Again, for an $A\in \GL_d(\R)$, we can define an invertible linear mar $\widehat{A}\colon\widehat{W}\mapsto\widehat{W}$ by setting for $u=u_1\otimes\cdots\otimes u_{d-1}$
$$
\widehat{A}(u_1\otimes\cdots\otimes u_{d-1})=(A^{\wedge1}u_1)\otimes\cdots\otimes(A^{\wedge(d-1)}u_{d-1})
$$
 and extending by linearity. It is easy to see that $\widehat{A}\colon W\mapsto W$ for $A\in GL_d(\R)$. Let us denote the restriction of the inner product $\langle\cdot,\cdot\rangle_{\wedge}$ and norm $\|\cdot\|_\wedge$ to $W$ by $\langle\cdot,\cdot\rangle_{W}$ and $\|\cdot\|_W$.

We say that $A\in \GL_d(\R)$ is \textit{fully proximal} if it has $d$ distinct eigenvalues in absolute value. Note that $A$ is fully proximal if and only if $A^{\wedge k}$ is $1$-proximal for every $k$ if and only if $\widehat{A}$ is $1$-proximal on $W$. We say that the tuple $\bA$ is fully proximal if there exists an $A\in\mathcal{S}(\bA)$ which is fully proximal.

We say that the tuple $\bA$ is \textit{fully strongly irreducible} or \textit{strongly irreducible over $W$} if there are no finite collection $V_1,\ldots,V_n$ of proper subspaces of $W$ such that
$$
\bigcup_{A\in\bA}\bigcup_{k=1}^n\widehat{A}V_k=\bigcup_{k=1}^nV_k.
$$

Before we prove Proposition~\ref{prop:multicond}, we need to recall two important tools. 

\begin{lemma}\label{lem:whoarestronglyirred}
	Suppose that $\bA$ is fully proximal and fully strongly irreducible then $\bA^\top=\{A_1^\top,\ldots,A_N^\top\}$ and $\bA^m=\{A_1\cdots A_m\}_{A_1,\ldots,A_m\in\bA}$ are also fully proximal and fully strongly irreducible for $m\geq1$. 
\end{lemma}

\begin{proof}
	Let $A\in \GL_d(\R)$ be a fully proximal matrix, and let $\lambda_1,\ldots,\lambda_d$ and $v_1,\ldots,v_d$ be the corresponding eigenvalues and eigenvectors. Then it is easy to see that any nonzero $w_i\in\mathrm{span}\{v_1,\ldots,v_{i-1},v_{i+1},\ldots,v_d\}^\perp$ is an eigenvector of $A^\top$ with eigenvalue $\lambda_i$. Indeed, $\langle A^{\top}w_i,v_j\rangle=\langle w_i,Av_j\rangle=\lambda_j\langle w_i,v_j\rangle=0$, we get that $A^\top w_i\in\mathrm{span}\{v_1,\ldots,v_{i-1},v_{i+1},\ldots,v_d\}^\perp$ and since $\dim\mathrm{span}\{v_1,\ldots,v_{i-1},v_{i+1},\ldots,v_d\}^\perp=1$ we have that $A^\top w_i=cw_i$ for some $c\in\R$. But since $c\langle w_i,v_i\rangle=\langle A^\top w_i,v_i\rangle=\langle w_i,Av_i\rangle=\lambda_i\langle w_i,v_i\rangle$ and $\langle w_i,v_i\rangle\neq0$, we get that $A^\top w_i=\lambda_iw_i$.
	
	Now, let us suppose that $\bA$ is not fully strongly irreducible and we show that then $\bA^\top$ is not fully strongly irreducible too. Let $V_1,\ldots,V_n$ be proper subspaces of $W$ such that
	$\bigcup_{A\in\bA}\bigcup_{i=1}^n\widehat{A}V_i=\bigcup_{i=1}^nV_i$. By the invertability, $\bigcup_{A\in\bA}\bigcup_{i=1}^n\widehat{A^{-1}}V_i=\bigcup_{i=1}^nV_i$. Clearly, $\widehat{A^{\top}}V^\perp=(\widehat{A^{-1}}V)^\perp$ for every proper subspace $V$ of $W$. So for any $i=1,\ldots,n$ and $A\in\bA$
	$$
	\widehat{A^\top}V_i^\top=(\widehat{A^{-1}}V_i)^\perp\subset\bigcup_{i=1}^n V_i^\perp,
	$$
	thus it follows that $\bA^\top$ is not fully strongly irreducible.
	
	Similarly, the fully proximality of $\bA$ implies clearly the fully proximality of $\bA^m$. Moreover, if $\bA^m$ is not fully strongly irreducible then there exists a finite family of proper subspaces $V_1,\ldots,V_n$ of $W$ such that $\bigcup_{A_1,\ldots,A_n\in\bA}\bigcup_{i=1}^n\widehat{A_1\cdots A_n}V_i=\bigcup_{i=1}^nV_i$. Thus, the tuple $\bA$ is not fully strongly irreducible for the family $\bigcup_{i=1}^n\bigcup_{k=0}^{m-1}\bigcup_{A_{1}\cdots A_{k}\in\bA}\{\widehat{A_{1}\cdots A_{k}}V_i\}$.
\end{proof}

Denote by $\mathcal{P}(W)$ the projective space of $W$ and $\SL(W)$ the space of linear maps of $W$ to $W$. Let 
$$
\mathcal{R}(\bA)=\left\{B\in \SL(W):\mathrm{rank}(B)=1\text{ and there exists $A_n\in\mathcal{S}(\bA)$ such that }\frac{\widehat{A_n}}{\|\widehat{A_n}\|_W}\to B\right\}
$$
and let
$$
\mathcal{L}(\bA)=\{BW:B\in\mathcal{R}(\bA)\}\subset\mathcal{P}(W).
$$
For a linear map $B\in \SL(W)$, denote by $\mathrm{Im}(B)$ the image space of $B$ and by
$\mathrm{Ker}(B)$ the kernel of $B$. Thus, $\mathcal{L}(\bA)=\{\mathrm{Im}(B):B\in\mathcal{R}(\bA)\}$.

Finally, let us denote the set of the fully proximal elements of the semigroup $\mathcal{S}(\bA)$ by
$\mathcal{S}_0(\bA)$. Then for every $A\in\mathcal{S}_0(\bA)$, the limit
$\lim_{n\to\infty}\frac{\widehat{A}^n}{\|\widehat{A}^n\|_W}$ exists and belongs to
$\mathcal{R}(\bA)$, and denote this limit by $\widehat{G}(A)$. Similarly, the limit
$G_k(A):=\lim_{n\to\infty}\frac{\left(A^{\wedge k}\right)^n}{\|\left(A^{\wedge k}\right)^n\|_k}$
exists and has rank 1. Moreover, $\mathrm{Im}(G_k(A))=\mathrm{span}\{v_1\wedge\ldots\wedge v_k\}$,
where $v_i$ is an eigenvector corresponding to the $i$-th largest eigenvalue in absolute value. Moreover, for a fully proximal matrix $A\in\mathcal{S}_0(\bA)$,
\begin{equation}\label{eq:relimages}
\mathrm{Im}(\widehat{G}(A))=\mathrm{Im}(G_1(A))\otimes\cdots\otimes\mathrm{Im}(G_{d-1}(A)).
\end{equation}

The following lemma is a corollary of Goldsheid and Guivarc'h \cite[Theorem~2.14]{GolGui}.

\begin{lemma}\label{lem:golgui}
	If $\bA$ is fully proximal and fully strongly irreducible then $\{\mathrm{Im}(\widehat{G}(A))\}_{A\in\mathcal{S}_0(\bA)}$ is dense in $\mathcal{L}(\bA)$ and $\bigcup_{A\in\bA}\widehat{A}\mathcal{L}(\bA)=\mathcal{L}(\bA)$.
\end{lemma}

\begin{lemma}\label{lem:strongirr}
	Let $\bA$ be a fully strongly irreducible matrix tuple. Then $\bA$ is strongly $k$-irreducible for every $k=1,\ldots,d-1$ and there exists no finite collection $V_1,\ldots,V_n$ of proper subspaces of $W$ such that $\mathcal{L}(\bA)\subseteq\bigcup_{i=1}^n\mathcal{P}(V_i)$.
\end{lemma}

\begin{proof}
	Let us argue by contradiction. First, suppose that $\bA$ is not strongly $k$-irreducible for some $k\in\{1,\ldots,d-1\}$. Let $V_1,\ldots,V_n$ be a finite collection of proper subspaces of $\wedge^k\R^d$ such that $\bigcup_{\ell=1}^n\bigcup_{A\in\bA}A^{\wedge k}V_k=\bigcup_{\ell=1}^nV_\ell$. Let 
	$$
	\widehat{V}_\ell:=\left(\wedge^1\R^d\otimes\cdots\otimes\wedge^{k-1}\R^d\otimes V_\ell\otimes\wedge^{k+1}\R^d\otimes\cdots\otimes\wedge^{d-1}\R^d\right)\cap W.
	$$
	It is easy to see that $\widehat{V}_\ell$ is a proper subspace of $W$ for all $\ell=1,\ldots,d-1$ and $\bigcup_{\ell=1}^n\bigcup_{A\in\bA}\widehat{A}\widehat{V}_k=\bigcup_{\ell=1}^n\widehat{V}_\ell$, which is a contradiction.
	
	Now, suppose that there exists a finite collection $V_1,\ldots,V_n$ of proper subspaces of $W$ such that $\mathcal{L}(\bA)\subseteq\bigcup_{i=1}^n\mathcal{P}(V_i)$. Without loss of generality, we may assume that $V_1,\ldots,V_n$ is minimal in the sense that $\mathcal{L}(\bA)\cap\mathcal{P}(V_i)$ is not contained in a finite union of subspaces of $V_i$. Indeed, if $\mathcal{L}(\bA)\cap\mathcal{P}(V_i)\bigcup\bigcup_{i=1}^{n'}\mathcal{P}(V_i')$ for a finite collection of proper subspaces $V_1',\ldots,V_{n'}'$ of $V_i$, then one can replace $V_i$ with $V_1',\ldots,V_{n'}'$. Clearly, the procedure terminates in finitely many steps.
	
	We will show that for every $A\in\bA$ and every $j\in\{1,\ldots,n\}$ there exists $i\in\{1,\ldots,n\}$ such that $\widehat{A}V_j=V_i$. Clearly,
	\[
	\begin{split}
		\widehat{A}\left(\mathcal{L}(\bA)\cap\mathcal{P}(V_j)\right)&\subseteq\bigcup_{i=1}^n\mathcal{P}(V_i),\\
		\widehat{A}\left(\mathcal{L}(\bA)\cap\mathcal{P}(V_j)\right)&\subseteq\mathcal{P}(\widehat{A}V_j).
	\end{split}
\]
Since $\widehat{A}$ is invertible on $W$ we get
$$
\mathcal{L}(\bA)\cap\mathcal{P}(V_j)\subseteq\mathcal{P}(V_j)\cap \bigcup_{i=1}^n\mathcal{P}(\widehat{A}^{-1}V_i)=\bigcup_{i=1}^n\mathcal{P}(\widehat{A}^{-1}V_i\cap V_j).
$$
But by the minimality assumption of $V_1,\ldots,V_n$, the subspace $\widehat{A}^{-1}V_i\cap V_j$ must be equal to $V_j$ for an $i\in\{1,\ldots,n\}$. 

Thus, $\bigcup_{\ell=1}^n\bigcup_{A\in\bA}\widehat{A}\widehat{V}_k=\bigcup_{\ell=1}^n\widehat{V}_\ell$, which is again a contradiction.
\end{proof}

\begin{proof}[Proof of Proposition~\ref{prop:multicond}]
	Let us argue by contradiction. Namely, there exists $s>0$ such that for every $C>0$ and $K\in\N$ there exist $\iii_{C,K},\jjj_{C,K}\in\Sigma_*$ such that for all $\kkk\in\Sigma_K$
	\[
	\phi^s(A_{\iii_{C,K}\kkk\jjj_{C,K}}) <
	C\phi^s(A_{\iii_{C,K}})\phi^s(A_{\jjj_{C,K}})
	\]
	We may first assume that $s\notin\N$, the proof of the integer case is similar and even simpler. For short, let $\lfloor s\rfloor=k$ and $\lceil s\rceil=k+1$
	By the singular value decomposition of $A_{\iii_{C,K}}$ and $A_{\jjj_{C,K}}$, let $u_1^{(C,K)},\ldots,u_d^{(C,K)}$ and $v_1^{(C,K)},\ldots,v_d^{(C,K)}$ be the orthonormal bases such that
	$\left\|A_{\jjj_{C,K}}u_i^{(C,K)}\right\|=\alpha_i(A_{\jjj_{C,K}})$ and $\left\|(A_{\iii_{C,K}})^\top v_i^{(C,K)}\right\|=\alpha_i(A_{\iii_{C,K}})$. Hence,
	\[
	\begin{split}
	\left\|A_{\jjj_{C,K}}^{\wedge j}\right\|_j&=\left\|A_{\jjj_{C,K}}^{\wedge j}u_1^{(C,K)}\wedge\cdots\wedge u_j^{(C,K)}\right\|_j,\\
	\left\|A_{\iii_{C,K}}^{\wedge j}\right\|_j&=\left\|\left(A_{\iii_{C,K}}^{\wedge j}\right)^\top v_1^{(C,K)}\wedge\cdots\wedge v_j^{(C,K)}\right\|_j,
	\end{split}
	\]
	for all $j=1,\ldots,d-1$. For short, let $u_{C,K}^{\wedge j}=u_1^{(C,K)}\wedge\cdots\wedge u_j^{(C,K)}$ and $v_{C,K}^{\wedge j}=v_1^{(C,K)}\wedge\cdots\wedge v_j^{(C,K)}$. So for every $C>0$ and $K\in\N$ and for all $\kkk\in\Sigma_K$
	\begin{multline*}
	\left\langle \left(A_{\iii_{C,K}}^{\wedge k}\right)^\top v_{C,K}^{\wedge k},A_{\kkk}^{\wedge k}A_{\jjj_{C,K}}^{\wedge k}u_{C,K}^{\wedge k}\right\rangle_k^{k+1-s}\left\langle \left(A_{\iii_{C,K}}^{\wedge k+1}\right)^\top v_{C,K}^{\wedge k+1},A_{\kkk}^{\wedge k+1}A_{\jjj_{C,K}}^{\wedge k+1}u_{C,K}^{\wedge k+1}\right\rangle_{k+1}^{s-k}\\
	<C\left\|A_{\jjj_{C,K}}^{\wedge k}u_{C,K}^{\wedge k}\right\|_k^{k+1-s}\left\|A_{\jjj_{C,K}}^{\wedge k+1}u_{C,K}^{\wedge k+1}\right\|_{k+1}^{s-k}\left\|\left(A_{\iii_{C,K}}^{\wedge k}\right)^\top v_{C,K}^{\wedge k}\right\|_k^{k+1-s}\left\|\left(A_{\iii_{C,K}}^{\wedge k+1}\right)^\top v_{C,K}^{\wedge k+1}\right\|_{k+1}^{s-k}.
	\end{multline*}
	By compactness and possibly taking a subsequence, we may assume that $\frac{A_{\jjj_{C,K}}u_i^{(C,K)}}{\alpha_i(A_{\jjj_{C,K}})}\to u_i^{(K)}$ and $\frac{(A_{\iii_{C,K}})^\top v_i^{(C,K)}}{\alpha_i(A_{\iii_{C,K}})}\to v_i^{(K)}$ as $C\to0$ and $\{u_1^{(K)},\ldots, u_d^{(K)}\}$ and $\{v_1^{(K)},\ldots, v_d^{(K)}\}$ are both orthonormal bases of $\R^d$. Hence, for every $K\in\N$ there exist orthonormal bases $\{u_1^{(K)},\ldots, u_d^{(K)}\}$ and $\{v_1^{(K)},\ldots, v_d^{(K)}\}$ of $\R^d$ such that for every $\kkk\in\Sigma_K$
	\begin{equation*}
	\left\langle v_{K}^{\wedge k},A_{\kkk}^{\wedge k}u_{K}^{\wedge k}\right\rangle_k\left\langle v_{K}^{\wedge k+1},A_{\kkk}^{\wedge k+1}u_{K}^{\wedge k+1}\right\rangle_{k+1}=0, 
	\end{equation*}
	where $u_{K}^{\wedge j}=u_1^{(K)}\wedge\cdots\wedge u_j^{(K)}$ and $v_{K}^{\wedge j}=v_1^{(K)}\wedge\cdots\wedge v_j^{(K)}$ for all $j=1,\ldots,d-1$. Finally, again by compactness let $K_n$ be the subsequence for which $\{u_1^{(K_n!)},\ldots, u_d^{(K_n!)}\}\to\{u_1^{*},\ldots, u_d^{*}\}$ and $\{v_1^{(K_n!)},\ldots, v_d^{(K_n!)}\to\{v_1^{*},\ldots, v_d^{*}\}$ and as $n\to\infty$, and $u_{*}^{\wedge j}=u_1^{*}\wedge\cdots\wedge u_j^{*}$ and $v_{*}^{\wedge j}=v_1^{*}\wedge\cdots\wedge v_j^{*}$.

	\begin{claim}
		There exists a fully proximal matrix $A\in\mathcal{S}_0(\bA)$ such that $u_{*}^{\wedge k}\notin\mathrm{Ker}(G_k(A))$ and $u_{*}^{\wedge k+1}\notin\mathrm{Ker}(G_{k+1}(A))$.
	\end{claim}

\begin{proof}[Proof of the Claim]
	Again, let us argue by contradiction, that is for every $A\in\mathcal{S}_0(\bA)$, $u_{*}^{\wedge k}\in\mathrm{Ker}(G_k(A))$ or $u_{*}^{\wedge k+1}\in\mathrm{Ker}(G_{k+1}(A))$. It is easy to see that $A^\top$ is also fully proximal and  $\mathrm{Ker}(G_k(A))^\perp=\mathrm{Im}(G_k(A^\top))$. Hence,
	$$
	\mathrm{Im}(G_k(A^\top))\in\mathcal{P}(\mathrm{span}\{u_{*}^{\wedge k}\}^\perp)\text{ or }\mathrm{Im}(G_{k+1}(A^\top))\in\mathcal{P}(\mathrm{span}\{u_{*}^{\wedge k+1}\}^\perp),
	$$
	and so by \eqref{eq:relimages}
	\[
	\begin{split}
	\bigcup_{A^\top\in\mathcal{S}_0(\bA^\top)}\{\mathrm{Im}(\widehat{G}(A^\top))\}&\subseteq\mathcal{P}\left(\left(\wedge^1\R^d\otimes\cdots\otimes\wedge^{k-1}\R^d\otimes\mathrm{span}\{u_{*}^{\wedge k}\}^\perp\otimes\wedge^{k+1}\R^d\otimes\cdots\otimes\wedge^{d-1}\R^d\right)\cap W\right)\\
	&\qquad\bigcup\mathcal{P}\left(\left(\wedge^1\R^d\otimes\cdots\otimes\wedge^{k}\R^d\otimes\mathrm{span}\{u_{*}^{\wedge k+1}\}^\perp\otimes\wedge^{k+2}\R^d\otimes\cdots\otimes\wedge^{d-1}\R^d\right)\cap W\right).
	\end{split}
	\]
	By Lemma~\ref{lem:golgui}, $\bigcup_{A^\top\in\mathcal{S}_0(\bA^\top)}\{\mathrm{Im}(\widehat{G}(A^\top))\}$ is dense in $\mathcal{L}(\bA^\top)$ and so by taking the closure,
	\[
	\begin{split}
		\mathcal{L}(\bA^\top)&\subseteq\mathcal{P}\left(\left(\wedge^1\R^d\otimes\cdots\otimes\wedge^{k-1}\R^d\otimes\mathrm{span}\{u_{*}^{\wedge k}\}^\perp\otimes\wedge^{k+1}\R^d\otimes\cdots\otimes\wedge^{d-1}\R^d\right)\cap W\right)\\
		&\qquad\bigcup\mathcal{P}\left(\left(\wedge^1\R^d\otimes\cdots\otimes\wedge^{k}\R^d\otimes\mathrm{span}\{u_{*}^{\wedge k+1}\}^\perp\otimes\wedge^{k+2}\R^d\otimes\cdots\otimes\wedge^{d-1}\R^d\right)\cap W\right),
	\end{split}
	\]
	which is a contradiction by Lemma~\ref{lem:strongirr}.
\end{proof}

Let $A_{\iii}\in\mathcal{S}_0(\bA)$ as in the Claim and let $m=|\iii|$. Now, let $\jjj\in\bigcup_{k=0}^\infty\Sigma_{k m}$ be arbitrary. Then for every  sufficiently large $n$, $q_n:=\frac{K_n!-|\jjj|}{m}\in\N$, and so
	\begin{multline*}
	0=\left\langle v_{K_n!}^{\wedge k},A_{\jjj}^{\wedge k}\frac{\left(A_{\iii}^{\wedge k}\right)^{q_n}}{\|\left(A_{\iii}^{\wedge k}\right)^{q_n}\|_k}u_{K_n!}^{\wedge k}\right\rangle_k\left\langle v_{K_n!}^{\wedge k+1},A_{\jjj}^{\wedge k+1}\frac{\left(A_{\iii}^{\wedge k+1}\right)^{q_n}}{\|\left(A_{\iii}^{\wedge k+1}\right)^{q_n}\|_{k+1}}u_{K_n!}^{\wedge k+1}\right\rangle_{k+1}\\
	\to\left\langle v_{*}^{\wedge k},A_{\jjj}^{\wedge k}G_k(A_{\iii}^{\wedge k})u_{*}^{\wedge k}\right\rangle_k\left\langle v_{*}^{\wedge k+1},A_{\jjj}^{\wedge k+1}G_{k+1}(A_{\iii}^{\wedge k+1})u_{*}^{\wedge k+1}\right\rangle_{k+1}
\end{multline*}
 as $n\to\infty$. In particular, for every $\jjj\in\bigcup_{k=0}^\infty\Sigma_{k m}$
 \[
 \left\langle v_{*}^{\wedge k},A_{\jjj}^{\wedge k}G_k(A_{\iii}^{\wedge k})u_{*}^{\wedge k}\right\rangle_k\left\langle v_{*}^{\wedge k+1},A_{\jjj}^{\wedge k+1}G_{k+1}(A_{\iii}^{\wedge k+1})u_{*}^{\wedge k+1}\right\rangle_{k+1}=0
 \] Since $G_k(A_{\iii}^{\wedge k})u_{*}^{\wedge k}$ and $G_{k+1}(A_{\iii}^{\wedge k+1})u_{*}^{\wedge k+1}$ are nonzero vectors, by Lemma~\ref{lem:kaenmor} we have that $\bA^m$ is neither strongly $k$-irreducible nor strongly $(k+1)$-irreducible, which together with Lemma~\ref{lem:whoarestronglyirred} and Lemma~\ref{lem:strongirr} is a contradiction.
\end{proof}

\begin{remark}
To show the claim of Proposition~\ref{prop:multicond} for a particular $s\notin\N\cap(0,d)$ (or $s\in\N\cap(0,d)$), it is enough to assume that the induced action of the tuple $\bA$ on $\wedge^{\lfloor s\rfloor}\R^d\otimes\wedge^{\lceil s\rceil}\R^d$ (on $\wedge^s\R^d$) is proximal and strongly irreducible. However, in our situation to prove Theorem~\ref{thm:mainTheorem} we need a condition which holds for every $s>0$ since the value $s_0$, for which $P(s_0)=\alpha(s_0)$, is unknown.
\end{remark}

\end{document}